\documentclass[a4paper,11pt]{article}

\usepackage{amssymb}
\usepackage{amsfonts,dsfont}
\usepackage{amsmath}
\usepackage{amsthm}
\usepackage{cite}
 \usepackage{enumitem}
\numberwithin{equation}{section}

\usepackage{xcolor}

\newcommand{\R}{\mathbb{R}}
\newcommand{\Z}{\mathbb{Z}}

\newcommand{\cuad}{{\sqcap\kern-.68em\sqcup}}

\newcommand{\dd}{\,d}

\newtheorem{definition}{Definition}[section]
\newtheorem{theorem}[definition]{Theorem}
\newtheorem{proposition}[definition]{Proposition}

\newtheorem{lemma}[definition]{Lemma}

\newtheorem{corollary}{Corollary}[section]

\renewcommand{\phi}{\varphi}

\newcommand{\cQ}{{\mathcal Q}}

\newcommand{\cT}{{\mathcal T}}

\newcommand{\bX}{{\mathbb X}}

 \allowdisplaybreaks

\begin{document}

\begin{center}
{\bf \large  Existence of positive Solutions to    H\'enon equation
\\[2mm]
on  integer Lattice graph and   Half-Space
}

\bigskip\bigskip

  {\small  Huyuan Chen\footnote{chenhuyuan@yeah.net, chenhuyuan@simis.cn} \qquad   Lixun Jia\footnote{lxjia26@m.fudan.edu.cn}
 \bigskip

  $ ^{1}$   Center for Mathematics and Interdisciplinary Sciences, Fudan University, \\
Shanghai 200433, China\\[8pt]

   $ ^{1,2}$ Shanghai Institute for Mathematics and Interdisciplinary Sciences,\\
 Shanghai 200433, China\\[12pt]

}
  \medskip

\begin{abstract}
 We study positive solutions of a discrete H\'enon equation on the integer lattice graph and on the half-space, with zero Dirichlet boundary conditions in the half-space case. The nonlinearity is governed by a H\'non-type potential with positive exponent. For sufficiently large nonlinearity exponents, we prove the existence of positive solutions both on the whole lattice in dimensions at least three and on the half-space in dimensions at least two.

 Thanks to the H\'enon potential, the standard variational compactness is unavailable to obtain the positive solution. Our  proofs combine weighted Banach spaces, two-sided Green kernel estimates and the Schauder fixed point theorem.
 \end{abstract}
 \end{center}
 {\small
   \smallskip

   \noindent {\small {\bf AMS Subject Classifications}:     35R02;  35A01; 35J62 .
  \smallskip

}
}
\bigskip

\setcounter{equation}{0}
\section{Introduction}
Let the integer lattice graph $\Z^d$, $d\geq 2$, be equipped with the counting measure and unit edge weight.    Every vertex has degree $2d$ and $x\sim y$ if and only if $\displaystyle \sum^d_{k=1}|x_k-y_k|=1$,  and
denote by $\Delta$  the Laplacian operator
$$\Delta u(x)=\sum_{y\sim x }\big(u(y)-u(x)\big),\quad\  \forall x\in\mathbb{Z}^d.$$
 Our aim is to consider the positive solution of the H\'enon equation
 \begin{equation}\label{eq 1.0}
 -\Delta u=Q_\beta u^{q}\quad {\rm in}\ \,\Omega,
 \end{equation}
 subject to the zero Dirichlet boundary condition $u=0$ on $\delta \Omega:=\{y\in\Z^d\setminus\Omega: \exists x\in\Omega, x\sim y\}$ if $\Omega\not=\Z^d$,
 where $q>1$, $\Omega$ is a unbounded domain of $\Z^d$, $Q_\beta\in C(\Omega)$ is a   H\'enon type potential satisfying
\begin{itemize}
			\item[ $(\cQ_0)$]  there exist $\beta>0$, $x_0\in\Omega$ and   $C>0$ such that
			\begin{equation}\label{con H-0}
		Q_\beta(x_0)=1\quad {\rm and} \quad 		0\leq  Q_\beta(x)\leq C(1+|x-x_0|)^{\beta}\ \ {\rm in}\ \, \Omega.
			\end{equation}
\end{itemize}

The H\'enon equation originates from H\'enon's investigation of the stability of spherically symmetric stellar systems~\cite{Henon}. To analyze the mass distribution and its evolution, he employed a concentric spherical shell model, leading to a nonlinear elliptic equation featuring a spatial weight function. In modern literature, the boundary-value problem
\[
\begin{cases}
-\Delta u=|x|^\alpha u^{p-1} & \text{in }\Omega,\\
u=0 & \text{on }\partial\Omega
\end{cases}
\]
is commonly referred to as the H\'enon equation. In 1982, Ni \cite{Ni} the H\'enon problem
\[
\begin{cases}
-\Delta u=|x|^\alpha u^{p-1} & \text{in } B_1(0),\\
u>0 & \text{in } B_1(0),\\
u=0 & \text{on }\partial B_1(0)
\end{cases}
\]
admits a radially symmetric solution for \(2<p<2_\alpha^*\) where $2_\alpha^*=\frac{2(N+\alpha)}{N-2}$,
  $N\geq3$ and $\alpha>0$. The crucial tool is the compact embedding
\[
H^1_{0,\mathrm{rad}}(B_1(0))
\hookrightarrow L^p(B_1(0),|x|^\alpha\,dx)
\quad\text{for }1\leq p<2_\alpha^*.
\]
 Later on, in the whole Euclidean space,  semilinear elliptic equations with H\'enon type potential
\[
-\Delta u=Q(x)u^{q} \qquad \text{in }\mathbb{R}^{d},
\]
have been studied  extensively.  Ni \cite{NiGeometry} investigated the
critical problem  motivated in part
by the prescribed scalar curvature problem; Gui
\cite{Gui1996} studied the existence and asymptotic behavior of
positive entire solutions.  Also see the references \cite{BaeChangPahk2002,GigaNgo2022} for the H\'enon type equations. \smallskip

In the discrete graph,   the discreteness leads to the failure of the dilation invariance, of admissibility of local singularity   and of applying classical ODE-based techniques in dealing with semilinear equations, such as the shooting method or phase-plane analysis.
 Moreover,  many different phenomenon of solutions structure appears for  semilinear equations in discrete graph,   see \cite{HLW,HLY21,LY20,LY22,G20,G18} and the references therein.  Particularly,  Grigor'yan-Lin-Yang \cite{GLY,GLY16}  considered Yamabe-type equations on graphs.    Hua and Xu obtained ground-state solutions for nonlinear Schr\"odinger equations
on lattice graphs \cite{HuaXu}. Hua-Li \cite{HL} proved the existence of externals for discrete Sobolev and
Hardy--Littlewood--Sobolev inequalities.  On  Liouville type theorem for the semilnear elliptic equations
\[
-\Delta_G u \geq Q(x) u^{q}\quad {\rm in}\ \, G
\]
has been studies \cite{GuHuangSun2023,MonticelliPunzoSomaglia2025,JleliSamet2026,NiuLu2025,HuaLinZhang2026,ChenHuaXu2026} and the references therein.

Recently, \cite{GHHS} established the nonexistence of positive solutions to the inequality
\[
-\Delta u \geq (1+|x|)^{\beta} u^q \quad \text{in } \Omega,
\]
for $\beta > -2$ and for exponents $q$ lying in the Serrin-subcritical and Serrin-critical regimes, where $\Omega$ denotes either $\mathbb{Z}^N$ or certain classes of unbounded domains. Taking the model $Q_\beta(x)=(1+|x|)^\beta$,   \cite{CHZ} investigated equation (\ref{eq 1.0}) and proved existence of positive solutions in the Sobolev-supercritical regime when $\beta \leq 0$ by the variational method, while establishing nonexistence in the Serrin-subcritical and Serrin-critical regimes for   $\beta\in\R$.
Precisely,  for $d\geq3$,
\begin{itemize}
			\item[ $(a)$] when $\Omega=\Z^d$ ,
(\ref{eq 1.0})  admits  a positive solution if $q>\frac{2(d+\beta)}{d-2}-1$ and $\beta\in(-2,0]$,  and no nonnegative nonzero solution if $q\in(0,\frac{d+\beta}{d-2}]$ and $\beta\in(-2,+\infty)$;
\item[ $(b)$]
when $\Omega=\Z^d_+$,
(\ref{eq 1.0})  admits  a positive solution if $q>\frac{2(d+\beta)}{d-2}-1$ and $\beta\in(-2,0]$,  and no nonnegative nonzero solution if $q\in(0,\frac{d+1+\beta}{d-1})$ and $\beta\in(-2,+\infty)$.
\end{itemize}
{\it There remain several open cases: $(i)$ the existence of positive solutions for $q$ lying strictly between the Serrin exponent and the Sobolev exponent when $\beta \in (-2, 0]$, and $(ii)$ the existence of positive solutions for $q$ exceeding the Serrin exponent when $\beta > 0$. }  Inspired by the second open gap, our purpose in this article is to consider positive solutions  to  equation (\ref{eq 1.0}) when $\beta>0$.

Firstly, we consider the H\'enon equation  in the whole $\mathbb{Z}^d$, i.e.
\begin{equation}\label{eq: 1.1}
 \begin{cases}
 -\Delta u(x)=Q_\beta(x)u(x)^q&\text{in } \mathbb Z^d,\\[1mm]
 \qquad u\geq 0&\text{in }\ \Z^d
 \end{cases}
\end{equation}
and our first results   state as follows.
\begin{theorem}\label{thm:intro-whole}
Let $d\geq3$ and $(\cQ_0)$ hold with $\beta>0$ and $x_0=0$. There exists
\[
 q^*=q^*(d,\beta)>\frac{d+\beta}{d-2}
\]
such that for every $q\geq q^*$, problem  (\ref{eq: 1.1})
has a positive solution satisfying
\[
 u(x)\asymp(1+|x|)^{2-d} \quad{\rm for}\ \,  x\in\Z^d.
\]
\end{theorem}
Our second purpose in this article is to prove the existence of solution of the H\'enon equation in the
half integer lattice graph
\begin{equation}\label{eq:intro-half-space}
 \Z^d_+=\{x=(x_1,x')\in \mathbb{Z}^d:x_1\geq1\},
\end{equation}
with zero Dirichlet data on its complement, where $d\ge2$.

\begin{samepage}
\begin{theorem}\label{thm:intro-half-equality}
Let $d\geq2$ and $(\cQ_0)$ hold with $\beta>0$ and $x_0=(1,0,\cdots,0)\in\Z^d_+$. There exists
\[
 q^{\sharp}=q^{\sharp}(d,\beta)>
 \frac{d+1+\beta}{d-1}
\]
such that for every $q\geq q^{\sharp}$,
\[
 \begin{cases}
 -\Delta u=Q_\beta u^q&\text{in }\mathbb{Z}_+^d,\\
 \qquad u=0&\text{on }\mathbb{Z}^d\setminus\mathbb{Z}_+^d
 \end{cases}
\]
has a positive solution satisfying
\[
 \overline{u}(x)\asymp\frac{x_1}{(1+|x|)^d}\quad{\rm for}\ \,  x\in\Z^d_+.
\]
\end{theorem}
\end{samepage}

Theorems \ref{thm:intro-whole} and \ref{thm:intro-half-equality} establish the existence of positive solutions to the discrete H\'enon equation with a H\'enon-type potential \(Q_\beta\) (\(\beta > 0\)) when the nonlinear exponent \(q\) exceeds a sufficiently large threshold. Crucially, for H\'enon-type potentials, the variational framework developed in \cite{CHZ} is inapplicable due to the failure of the Palais-Smale (PS) condition. To overcome this obstruction, we employ fixed-point theory instead. Specifically, we introduce weighted Banach spaces \(\mathbb{X}_{g_d}\) and \(\mathbb{X}_{g_{d,+}}\) defined on \(\Omega = \mathbb{Z}^d\) and \(\mathbb{Z}^d_+\), respectively, where the weights \(g_d\) and \(g_{d,+}\) are normalized Green functions (or, in dimension two, the potential kernel in the whole-space case). The core analytical tools are Proposition \ref{lem:eta-q} and Proposition \ref{lem:half-eta-q}, which yield, for \(q \geq q^* > \frac{d+\beta}{d-2}\),
\begin{equation}\label{es-1}
\frac{1}{\Phi_{d}(0,0)} \left\| \Phi_{d} * \big(Q_\beta \, g_{d}^q \, \mathbf{1}_{\mathbb{Z}_+ \setminus \{0\}} \big) \right\|_{g_{d}} \leq (q-1)^{q-1} q^{-q},
\end{equation}
and, for \(q \geq q^\# > \frac{d+1+\beta}{d-1}\),
\begin{equation}\label{es-1+}
\frac{1}{\Phi_{d,+}(e_1,e_1)} \left\| \Phi_{d,+} * \big(Q_\beta \, g_{d,+}^q \, \mathbf{1}_{\mathbb{Z}_+^d \setminus \{e_1\}} \big) \right\|_{g_{d,+}} \leq (q-1)^{q-1} q^{-q}.
\end{equation}
These estimates ensure that the associated convolution-type operators are continuous and compact on bounded subsets of the respective weighted spaces. Leveraging this, we construct closed, bounded, convex subsets \(\mathcal{K}_q \subset \mathbb{X}_{g_d}\) and \(\mathcal{K}_{q,+} \subset \mathbb{X}_{g_{d,+}}\) and apply the Schauder fixed point theorem to obtain normalized fixed points. A subsequent scaling argument then recovers positive solutions to the original equations (\ref{eq 1.0}).

Note that our method is based on the decays at infinity of the Green kernel.  In $\Z^2$, the Green function diverges at infinity,
which prevents the existence  for the H\'enon equation. While,  in the half-space $\Z^2_+$, we can show that the  Green function decays
at infinity, which guarantees the existence of positive solution for large $q$.
  We emphasize that \(q^*\) and \(q^\#\) are not sharp thresholds but merely sufficient lower bounds guaranteeing the validity of (\ref{es-1}) and (\ref{es-1+}), respectively. Consequently, the existence of positive solutions remains open for the critical gaps \(q \in \big(\frac{d+\beta}{d-2},\, q^*\big]\) in \(\mathbb{Z}^d\) and \(q \in \big(\frac{d+1+\beta}{d-1},\, q^\#\big]\) in \(\mathbb{Z}^d_+\).

Our results are complementary to the variational existence results for \(\beta\le0\) and to the nonexistence results in the Serrin-subcritical and Serrin-critical regimes, and they provide a partial answer to the open gap  stated in the introduction concerning the case \(\beta>0\) and \(q\) beyond the Serrin exponent.\smallskip

The remainder of this paper is organized as follows. Section~2 is devoted to preliminary material.  We collect basic estimates on the whole lattice \(\mathbb Z^d\), including the Green function bounds, the weighted space \(\mathbb X_{g_d}\), and the key estimates in Proposition \ref{lem:eta-q} and Proposition \ref{lem:half-eta-q}.  Also we develop the corresponding estimtes on the half-space \(\mathbb Z_+^d\).  In Section 3, we   apply the Schauder fixed point theorem in a suitable convex set to obtain positive solutions of (\ref{eq 1.0})  in the whole lattice \(\mathbb Z^d\) with \(d\ge3\) and \(\mathbb Z^d_+\) with \(d\ge2\).   Finally, the appendix is devoted to the construction and two-sided estimates of the two-dimensional Dirichlet Green kernel on \(\mathbb Z_+^2\), which is used in Proposition  \ref{prop:two-dimensional-green}.

 \section{Important estimates}
 Notations: $|x|=\sqrt{\sum^d_{i=1}x_i^2}$, which denotes the Euclidean norm of $x\in\mathbb Z^d$;
  $A\asymp B$ means that $\frac1CA\leq B\leq CA$ for some universal constant $C\geq 1$. In the subsequent proof, we set $C$ in $(\cQ_0)$ to 1, which is not essential to the results or arguments presented in this paper.

The following maximum principle  is in the discrete setting.
 \begin{theorem}\label{thm:max}\cite[Theorem 2.1]{CH}
     For a graph $G=(V,E)$ and a (possibly infinite) connected subset $\Omega\subset V$   verifying
   either $\delta \Omega\not=\emptyset$ or $\Omega$ is unbounded,
     if  $u: \overline{\Omega}\to \R$ satisfies
\begin{equation}\label{eq 2.1 cm}
\left\{
\begin{array}{lll}
-\Delta u+c u     \geq  0   \quad
   &{\rm in}\ \  \Omega , \\[1mm]
 \phantom{ ----  }
u\geq 0 \quad &{\rm   in}\ \ \,    \delta\Omega, \\[1mm]
\liminf\limits_{d(x,0)\to \infty, x\in \Omega}u(x)\geq 0,
 \end{array}
 \right.
\end{equation} where $c:\Omega\to[0,\infty),$
then $u\geq 0$ in $\Omega$.  Furthermore, $u\equiv 0$ in $\Omega$ or $u>0$ in $\Omega$.
 \end{theorem}

 \subsection{Basic estimates in $\Z^d$ }

 Let $\Phi_d$ be the positive fundamental solution of $-\Delta$ on
$\Z^d$, $d\ge3$,
$$
 \begin{cases}
\qquad -\Delta\Phi_d=\delta_0&\text{in } \mathbb Z^d,\\[1mm]
\displaystyle \lim_{|x|\to\infty}\Phi_d(x)=0,
 \end{cases}
 $$
where $\delta_0$ is the Dirac mass at the origin.

The lattice Green-function estimates in Chapter~4 of \cite{LL} and
(4.10)--(4.11) of \cite{CHZ} give
\begin{equation}\label{eq:whole-green}
 c(1+|x|)^{2-d}\leq\Phi_d(x)
 \leq C(1+|x|)^{2-d}\quad{\rm for}\ \,  x\in\Z^d.
\end{equation}
Let
\[g_d(x)=\frac{\Phi_d(x)}{\Phi_d(0)}\quad {\rm for}\ \, x\in\Z^d.\]
Maximum principle implies that
$$ g_d(0)=1,\quad  g_d(x)>0\quad {\rm  and}\quad
 g_d(x)\asymp  |x|^{2-d} \ \ {\rm as} \ \, |x|\to+\infty.$$

{\it We claim that \(g_d(x)<1\) for every \(x\neq 0\).}  Indeed, choose \(R\) large enough that \(g_d<1\) outside
\(B_R\), and apply the discrete maximum principle to
\(B_R\setminus\{0\}\). Since \(g_d(0)=1\), this gives \(g_d\leq 1\) in
\(B_R\). If \(g_d(x)=1\) at some \(x\neq 0\), then the mean-value property
forces \(g_d=1\) at every neighbor of \(x\). Repeating this argument along
a path avoiding the origin would give \(g_d=1\) at points arbitrarily far
from the origin, contradicting \(g_d(x)\to 0\). Therefore,
\[
    0<g_d(x)<1,\qquad x\neq 0.
\]
Now we set
\begin{equation}\label{eq: 1}
 \rho_d=\sup_{x\in\Z^d\setminus\{0\}}g_d(x)<1
\end{equation}
and
\begin{equation}\label{norm-1}
 \|v\|_{g_d}=\sup_{x\in\Z^d}\frac{|v(x)|}{g_d(x)},
 \qquad
 \bX_{g_d}=\{v\in C(\Z^d):\|v\|_{g_d}<\infty\}.
\end{equation}

\begin{proposition}\label{lem:eta-q}
Let $d\geq3$ and $(\cQ_0)$ hold with $\beta>0$ and $x_0=0$,  and for   $q>\frac{d+\beta}{d-2}$, denote
\begin{equation}\label{eq:whole-eta-c}
 \eta_q=\frac1{\Phi_d(0)}
 \left\| \Phi_d\ast  \left(
 Q_\beta g_d^q\mathbf 1_{\Z^d\setminus\{0\}}
 \right)\right\|_{g_d},
 \qquad
 c_q=\frac{(q-1)^{q-1}}{q^q}.
\end{equation}
Then   there exists $q^*>\frac{d+\beta}{d-2}$ such that
\begin{equation}
 \eta_q\leq c_q\qquad\text{for every }q\geq q^*.
\end{equation}
\end{proposition}

\begin{lemma}\label{lm 3.1}
 Let  $d\geq 1$,  $s>d$ and $0<a\leq d$. Then for all $x\in\Z^d$,
\begin{equation}\label{eq:weighted-convolution}
 \sum_{y\in\Z^d}(1+|x-y|)^{-a}(1+|y|)^{-s}
 \leq C(1+|x|)^{-a},
\end{equation}
where $C>0$ depends only on $d$, $a$, and $s$.
\end{lemma}

\begin{proof}
We first record two elementary estimates. Since
\[
 |\{z\in\Z^d:j\leq|z|<j+1\}|\leq C(1+j)^{d-1},
\]
the sum $\sum_z(1+|z|)^{-s}$ converges, and for every $R\geq4$,
\begin{equation}\label{eq:lattice-shell-sum}
 \sum_{|z|\leq R}(1+|z|)^{-a}
 \leq C\begin{cases}
 (1+R)^{d-a},&0<a<d,\\[1mm]
 \log(2+R),&a=d.
 \end{cases}
\end{equation}

We now prove \eqref{eq:weighted-convolution}. If $|x|\leq 8$, then
$(1+|x|)^{-a}>9^{-a}$ and $(1+|x-y|)^{-a}\leq1$, so
\[
 \sum_y(1+|x-y|)^{-a}(1+|y|)^{-s}
 \leq\sum_y(1+|y|)^{-s}
 \leq C
 \leq C9^a (1+|x|)^{-a}.
\]
Thus \eqref{eq:weighted-convolution} holds in this case.

Assume now that $|x|>8$ and $R=\frac12 |x|$. Split the sum into the two regions
$|x-y|\geq \frac{R}{2}$ and $|x-y|<\frac{R}{2}$. On the first region,
\[
 \sum_{|x-y|\geq \frac{R}{2}}
 (1+|x-y|)^{-a}(1+|y|)^{-s}
 \leq CR^{-a}\sum_y(1+|y|)^{-s}
 \leq CR^{-a}.
\]
On the second region, we have $|y|>\frac{R}{2}$. Setting $z=x-y$ and using
\eqref{eq:lattice-shell-sum}, we obtain
\begin{align*}
 \sum_{|x-y|<\frac{R}{2}}
 (1+|x-y|)^{-a}(1+|y|)^{-s}
 &\leq CR^{-s}\sum_{|z|<R/2}(1+|z|)^{-a}\\
 &\leq C
 \begin{cases}
  R^{d-a-s},&0<a<d,\\[1mm]
  R^{-s}\log(2+R),&a=d.
 \end{cases}
\end{align*}
Since $s>d$, both expressions are bounded by $CR^{-a}$.
Combining the estimates for the two regions gives
\eqref{eq:weighted-convolution}.
\end{proof}
 From \eqref{eq:whole-green} and
\eqref{eq:weighted-convolution}, we have the following corollary
\begin{corollary}\label{cor:weighted-green}
Let $d\geq 3$,  $s>d$, then there exists $C=C(d,s)>0$ such that
\begin{equation}\label{eq:whole-weighted-green}
 \sum_{y\in\Z^d}\Phi_d(x-y)(1+|y|)^{-s}
 \leq C(1+|x|)^{2-d}.
\end{equation}
\end{corollary}

\begin{proof}[Proof of Proposition \ref{lem:eta-q}.]
Fix $q_0>\frac{d+\beta}{d-2}$, then  we have $\eta_{q_0}<\infty$. For $y\neq 0$, we have
\[
 g_d(y)^q
 =g_d(y)^{q_0}g_d(y)^{q-q_0}
 \leq
 \rho_d^{\,q-q_0}g_d(y)^{q_0}.
\]
So we have
\[
 0\leq
 \Phi_d\ast \!\left(
  Q_\beta g_d^q\mathbf 1_{\Z^d\setminus\{0\}}
 \right)
 \leq
 \rho_d^{\,q-q_0}
 \Phi_d\ast \!\left(
  Q_\beta g_d^{q_0}\mathbf 1_{\Z^d\setminus\{0\}}.
 \right)
\]
 By Corollary \ref{cor:weighted-green} and assumption $(\cQ_0)$,  we derive  that
$$0<\Phi_d\ast \!\left(
  Q_\beta g_d^{q_0}\mathbf 1_{\Z^d\setminus\{0\}}
 \right)(x) \leq C(d,q_0)(1+|x|)^{2-d}<+\infty,  $$
hence  $\eta_{q_0}<+\infty$ and
\begin{equation}\label{eq:whole-eta-decay}
 \eta_q\leq\eta_{q_0}\rho_d^{q-q_0}.
\end{equation}
On the other hand, for $q>1$, we have $1-t>0$ for every
$t\in[0,\frac{1}{q}]$. Moreover,
\[
\frac{d}{dt}\bigl[-\ln(1-t)\bigr]=\frac{1}{1-t}.
\]
Therefore, by the fundamental theorem of calculus,
\[
\begin{aligned}
-\ln\left(1-\frac{1}{q}\right)
&=\int_0^{\frac{1}{q}}\frac{dt}{1-t}.
\end{aligned}
\]
Note that
\[
\frac{1}{1-t}\le \frac{q}{q-1}\quad \text{for every $t\in[0,\frac{1}{q}]$}.
\]
Consequently,
\begin{align*}
-\ln\left(1-\frac{1}{q}\right)
&=\int_0^{1/q}\frac{dt}{1-t} \le \int_0^{1/q}\frac{q}{q-1}\,dt= \frac{1}{q-1}.
\end{align*}
 Multiplying both sides by $q-1>0$, we obtain
\[
(q-1)\ln\left(1-\frac{1}{q}\right)\ge -1,
\]
which leads to
\[
\left(1-\frac{1}{q}\right)^{q-1}\ge e^{-1}.
\]
Therefore,
\[
c_q
=\frac{1}{q}\left(1-\frac{1}{q}\right)^{q-1}
\ge\frac{1}{eq}.
\]

Since $\rho_d<1$ and $q>q_0$, we have
\[
\frac{\eta_q}{c_q}
\leq
e q\,\eta_{q_0}\rho_d^{q-q_0}
\rightarrow0
\qquad\text{as }q\to\infty,
\]
so there exists
$q^*>q_0$ such that
\begin{equation}\label{eq:whole-testing}
 \eta_q\leq c_q\qquad\text{for every }q\geq q^*.
\end{equation}
We complete the proof. \end{proof}

\subsection{Basic estimates in $\Z^d_+$}
We first consider $d\geq3$. For
$y=(y_1,y')\in\mathbb Z_+^d$, put
$y^*=(-y_1,y')$. Define
\begin{equation}\label{eq:half-reflection}
\Phi_{d,+}(x,y)=
\begin{cases}
\Phi_d(x-y)-\Phi_d(x-y^*),
   &x\in\mathbb Z_+^d,\\[1mm]
0,&x\in\mathbb Z^d\setminus\mathbb Z_+^d,
\end{cases}
\end{equation}
where we recall that
$$\mathbb{Z}^d_+=\{x=(x_1,x')\in\mathbb Z\times \Z^{d-1}: x_1>0\}. $$
By the reflection principle, this is the Dirichlet
Green kernel on $\mathbb Z_+^d$. In particular,
\[
\Phi_{d,+}(x,y)>0,
\qquad x,y\in\mathbb Z_+^d,
\]
and
\[
-\Delta_{x}\Phi_{d,+}(x,y)=\delta_y(x),
\qquad x,y\in\mathbb Z_+^d.
\]
It follows by Theorem~9.3 of \cite{GHHS} that
\begin{equation}\label{eq:half-green-global}
 \Phi_{d,+}(x,y)\asymp
 (1+|x-y|)^{2-d}
 \left(1\wedge\frac{x_1y_1}{(1+|x-y|)^2}\right),\quad x,y\in \mathbb{Z}^d_+.
\end{equation}

In particular,
\begin{equation}\label{eq:half-green-upper}
\Phi_{d,+}\leq Cx_1y_1(1+|x-y|)^{-d}.
\end{equation}

When $d=2$, Let $A$ be the potential kernel defined in
\eqref{eq:two-potential-definition}. For
$y=(y_1,y_2)\in\mathbb Z_+^2$, put
$y^*=(-y_1,y_2)$ and define
\[
\Phi_{2,+}(x,y)=
\begin{cases}
A(x-y^*)-A(x-y),&x\in\mathbb Z_+^2,\\[1mm]
0,&x\notin\mathbb Z_+^2.
\end{cases}
\]
Then we have the following Green kernel estimates on half space.

\begin{proposition}\label{prop:two-dimensional-green}
  There are constants $C>c>0$ such that
\begin{equation}\label{eq:two-green-logarithmic}
 c\log\!\left(1+\frac{x_1y_1}{1+|x-y|^2}\right)
 \leq \Phi_{2,+}(x,y)
 \leq C\log\!\left(1+\frac{x_1y_1}{1+|x-y|^2}\right)\quad {\rm for}\ x,y\in\Z^2_{+}
\end{equation}
and
$$
\Phi_{2,+}(x,y)=\Phi_{2,+}(y,x)\quad {\rm for}\ x,y\in\Z^2_{+}.
$$
Particularly,
\begin{equation}\label{eq:two-pole-appendix}
 \frac{\Phi_{2,+}(x,e_1)}{\Phi_{2,+}(e_1,e_1)}
 \asymp\frac{x_1}{(1+|x|)^2}.
\end{equation}
\end{proposition}
The proof of Proposition \ref{prop:two-dimensional-green} is addressed {\it in the Appendix}.

For \(d=2\), estimate~\eqref{eq:two-green-logarithmic} gives
\[
\begin{aligned}
\Phi_{2,+}(x,y)
&\le C\log\left(
1+\frac{x_1y_1}{1+|x-y|^2}
\right)\\
&\le C\frac{x_1y_1}{1+|x-y|^2}\\
&\le 2C\frac{x_1y_1}{(1+|x-y|)^2},
\qquad \forall\, x,y\in\mathbb Z_+^2.
\end{aligned}
\]
Here we used \(\log(1+t)\le t\) for \(t\ge0\) and
\[
(1+r)^2\le 2(1+r^2),
\qquad r\ge0.
\]
Consequently, after adjusting the constant,
the upper bound~\eqref{eq:half-green-upper}
holds for every \(d\ge2\).

Let
\[ g_{d,+}(x)=\frac{\Phi_{d,+}(x,e_1)}{\Phi_{d,+}(e_1,e_1)},\quad x\in \mathbb{Z}^d_+.\]
Maximum principle implies that
$$
g_{d,+}(e_1)=1,\quad g_{d,+}(x)>0,\quad {\rm  and}\quad
 g_{d,+}\asymp \frac{x_1}{(1+|x|)^d} \ \ {\rm as} \ \, |x|\to+\infty.
$$
By the maximum principle,  we derive that
$$
0<g_{d,+}(x)<1\quad{\rm for}\ x\in \mathbb{Z}^d_+\setminus\{e_1\}.
$$
Now we set
\begin{equation}\label{eq:half-rho}
\rho_{d,+}=\sup_{x\in\mathbb{Z}^d_+\setminus\{e_1\}}g_{d,+}(x)<1
\end{equation}
and
\begin{equation}\label{norm-2}
 \|v\|_{g_{d,+}}=\sup_{x\in\Z^d_+}\frac{|v(x)|}{g_{d,+}(x)},
 \qquad
 \bX_{g_{d,+}}=\big\{v\in C(\Z^d_+):\|v\|_{g_{d,+}}<+\infty\big\}.
\end{equation}

\begin{proposition}\label{lem:half-eta-q}
Let $d\geq2$ and $(\cQ_0)$ hold with $\beta>0$ and $x_0=(1,0\cdots,0)$. For  $q>\frac{d+1+\beta}{d-1}$,  denote
\begin{equation}\label{eq:half-eta-c}
\eta_{q,+}=
 \frac1{\Phi_{d,+}(e_1,e_1)}
 \left\|\Phi_{d,+}*\left(Q_\beta g_{d,+}^q\mathbf 1_{\Z_+^d\setminus\{e_1\}}\right)\right\|_{g_{d,+}},
 \qquad
 c_q=\frac{(q-1)^{q-1}}{q^q}.
\end{equation}
Then there exists $q^{\sharp}>\frac{d+1+\beta}{d-1}$ such that
\[
\eta_{q,+}\leq c_q\qquad\text{for }\ q\geq q^{\sharp}.
\]
\end{proposition}

\begin{lemma}\label{lem:half-weighted-green}
Let $d\geq 2$,  $s>d+1$ and for all $x\in \mathbb{Z}^d_+$, we have
\[
 \sum_{y\in\Z^d_+}\Phi_{d,+}(x,y)(1+|y|)^{-s}\le Cx_1(1+|x|)^{-d}\leq Cg_{d,+}(x).
\]
\end{lemma}
\begin{proof}
It follows from \eqref{eq:weighted-convolution}, \eqref{eq:half-green-upper} and
$y_1\leq1+|y|$, for $s>d+1$, we have
\begin{equation}\label{eq:half-weighted-green}
 \begin{aligned}
 \sum_{y\in\Z^d_+}\Phi_{d,+}(x,y)(1+|y|)^{-s}
 &\leq Cx_1\sum_{y\in\Z^d}
 (1+|x-y|)^{-d}(1+|y|)^{-(s-1)}\\
 &\leq Cx_1(1+|x|)^{-d}\leq Cg_{d,+}(x).
 \end{aligned}
\end{equation}
We complete the proof. \end{proof}

\begin{proof}[Proof of Proposition \ref{lem:half-eta-q}]
Fix $q_{1}>\frac{d+1+\beta}{d-1}$, then we have $\eta_{q_1,+}<\infty$. For $y\neq e_1$ we have
\[
 g_{d,+}(y)^q
 =g_{d,+}(y)^{q_{1}}g_{d,+}(y)^{q-q_{1}}
 \leq\rho_{d,+}^{q-q_{1}}g_{d,+}(y)^{q_{1}}.
\]
So we have
\[
 0\leq
 \Phi_{d,+}*\left(Q_\beta g_{d,+}^q\mathbf 1_{\Z_+^d\setminus\{e_1\}}\right)
 \leq
 \rho_{d,+}^{\,q-q_1}\Phi_{d,+}*\left(Q_\beta g_{d,+}^{q_1}\mathbf 1_{\Z_+^d\setminus\{e_1\}}\right)
\]
and
\begin{equation}\label{eq:half-eta-decay}
\eta_{q,+}\leq\eta_{q_1,+}\rho_{d,+}^{q-q_1}.
\end{equation}
Moreover, by the same method in Proposition \ref{lem:eta-q} and inequality in Lemma \ref{lem:half-weighted-green},  we have $c_q\geq\frac{1}{eq}$ and
\[
\frac{\eta_{q,+}}{c_q}
\leq
e q\,\eta_{q_1,+}\rho_{d,+}^{q-q_1}
\rightarrow0\qquad\text{as }q\to\infty.
\]
So there exists $q^{\sharp}>q_{1}$ such that
\begin{equation}\label{eq:half-testing}
\eta_{q,+}\leq c_q\qquad\text{for every }q\ge q^{\sharp}.
\end{equation}
We complete the proof. \end{proof}

\section{Positive solutions for H\'enon equations }

\subsection{Existence in $\Z^d$ with $d\geq 3$ }

\begin{lemma}\label{lem:compactness-Rq}
For \(q\ge q^*\), denote
\[
  \mathcal R_q(w)
  = g_d * \left(Q_\beta |w|^q \mathbf 1_{\mathbb Z^d\setminus\{0\}}\right)
  \qquad\text{for } w\in \mathbb X_{g_d}.
\]
Then \(\mathcal R_q\) is continuous on \(\mathbb X_{g_d}\) and maps every bounded subset of \(\mathbb X_{g_d}\) into a relatively compact subset.
\end{lemma}

\begin{proof}
Recall that
\[
  g_d=\frac{\Phi_d}{\Phi_d(0)},
\]
so for any function \(f\) on \(\mathbb Z^d\),
\begin{equation}\label{eq:g-Phi-convolution}
  g_d * f
  =
  \frac{1}{\Phi_d(0)}\,\Phi_d * f.
\end{equation}
Put
\[
  s=q(d-2)-\beta>d.
\]
The inequality \(s>d\) is equivalent to
\(q>\frac{d+\beta}{d-2}\).

If \(\|w\|_{g_d}\le M\), then \(|w(y)|\le M g_d(y)\), and hence,
using \(g_d(y)\asymp(1+|y|)^{2-d}\),
\begin{equation}\label{eq:source-bound}
  Q_\beta(y)|w(y)|^q
  \le
  C_{M} (1+|y|)^{\beta+q(2-d)}
  =
  C_{M} (1+|y|)^{-s}.
\end{equation}

Choose \(s'\) with \(d<s'<s\). For \(|y|>R\),
\[
  (1+|y|)^{-s}
  \le
  (1+R)^{-(s-s')}(1+|y|)^{-s'}.
\]
Applying Corollary  \ref{cor:weighted-green} with \(s'>d\), we obtain
\begin{equation}\label{eq:uniform-tail}
  \sup_{\|w\|_{g_d}\le M}
  \left\|
  g_d * \left(Q_\beta |w|^q \mathbf 1_{\{|y|>R\}}\right)
  \right\|_{g_d}
  \le
  C_{M} (1+R)^{-(s-s')}
  \rightarrow0
  \qquad {\rm as}\ R\to\infty.
\end{equation}

For each fixed \(y\), by \eqref{eq:whole-green} and
\(1+|x|\le(1+|x-y|)(1+|y|)\),
\[
  \frac{\Phi_d(x-y)}{g_d(x)}
  \le
  C\Big(\frac{1+|x|}{1+|x-y|}\Big)^{d-2}
  \le
  C(1+|y|)^{d-2}.
\]
Therefore
\begin{equation}\label{eq:Phi-in-Xg}
  \|\Phi_d(\cdot-y)\|_{g_d}
  \le
  C(1+|y|)^{d-2}
  <\infty,
\end{equation}
so \(\Phi_d(\cdot-y)\in\mathbb X_{g_d}\).

For \(w\in\mathbb X_{g_d}\), define the operator
\[
  \mathcal R_{q,R}(w)(x)
  =
  \frac{1}{\Phi_d(0)}
  \sum_{0<|y|\le R}
  \Phi_d(x-y)Q_\beta(y)|w(y)|^q.
\]

\medskip
\noindent\textit{We next show that \(\mathcal R_{q,R}\) is continuous.}
Suppose \(w_n\to w\) in \(\mathbb X_{g_d}\). Then for every fixed \(y\)
with \(0<|y|\le R\),
\[
  |w_n(y)-w(y)|
  \le
  g_d(y)\|w_n-w\|_{g_d}
  \rightarrow0.
\]
Since \(\{w_n\}\) is bounded in \(\mathbb X_{g_d}\), there exists \(M>0\)
such that
\[
  \|w_n\|_{g_d}\le M
  \ {\rm and}\
  \|w\|_{g_d}\le M.
\]
By the mean value theorem applied to \(t\mapsto t^q\) on \([0,\infty)\),
\[
  \bigl||w_n(y)|^q-|w(y)|^q\bigr|
  \le
  qM^{q-1}g_d(y)^q\|w_n-w\|_{g_d}.
\]
Since \(\|\Phi_d(\cdot-y)\|_{g_d}<\infty\) by \eqref{eq:Phi-in-Xg}
and \(\{y:0<|y|\le R\}\) is finite,
\[
\begin{aligned}
  \|\mathcal R_{q,R}(w_n)-\mathcal R_{q,R}(w)\|_{g_d}
  &\le
  \frac{1}{\Phi_d(0)}
  \sum_{0<|y|\le R}
  Q_\beta(y)\|\Phi_d(\cdot-y)\|_{g_d}
  \bigl||w_n(y)|^q-|w(y)|^q\bigr|\\
  &\le
  \frac{qM^{q-1}}{\Phi_d(0)}
  \Big(
    \sum_{0<|y|\le R}
    Q_\beta(y)g_d(y)^q\|\Phi_d(\cdot-y)\|_{g_d}
  \Big)
  \|w_n-w\|_{g_d}.
\end{aligned}
\]
The sum in parentheses is finite and independent of \(n\), because
\[
  Q_\beta(y)g_d(y)^q\|\Phi_d(\cdot-y)\|_{g_d}
  \le
  C(1+|y|)^{-s}(1+|y|)^{d-2}
  =
  C(1+|y|)^{-s+d-2},
\]
and \(-s+d-2<-2<0\) since \(s>d\). Hence
\[
  \|\mathcal R_{q,R}(w_n)-\mathcal R_{q,R}(w)\|_{g_d}
  \rightarrow0\quad {\rm as}\ n\to\infty.
\]
Thus \(\mathcal R_{q,R}\) is continuous.

\medskip
\noindent\textit{We next prove the compactness of \(\mathcal R_{q}\).}
For every \(w\in\mathbb X_{g_d}\), let
\[
  \mathcal R_{q,R}(w)
  =
  \sum_{0<|y|\le R}
  \frac{Q_\beta(y)|w(y)|^q}{\Phi_d(0)}
  \Phi_d(\cdot-y)
  \in V_R,
\]
where
\[
  V_R
  :=
  \operatorname{span}
  \left\{
    \Phi_d(\cdot-y):
    y\in\mathbb Z^d,\ 0<|y|\le R
  \right\}.
\]
Let $B_M=\{w:\|w\|_{g_d}\le M\}\subset\mathbb X_{g_d}$. Then
\[
  |w(y)|\le M g_d(y)
  \qquad
  \text{for every } w\in B_M,\ y\in\mathbb Z^d,
\]
and therefore
\[
  \left|
    \frac{Q_\beta(y)|w(y)|^q}{\Phi_d(0)}
  \right|
  \le
  \frac{M^q}{\Phi_d(0)}Q_\beta(y)g_d(y)^q\quad \text{for every } w\in B_M,\ y\in\mathbb Z^d.
\]
For fixed \(R\), the set \(\{y:0<|y|\le R\}\) is finite, so
\(\mathcal R_{q,R}(B_M)\) is bounded in the finite-dimensional space
\(V_R\). Since every bounded subset of a finite-dimensional space is
relatively compact, \(\mathcal R_{q,R}(B_M)\) is relatively compact in
\(\mathbb X_{g_d}\). Hence \(\mathcal R_{q,R}\) is a compact operator.

By \eqref{eq:uniform-tail},
\[
  \lim_{R\to\infty}
  \sup_{w\in B_M}
  \|\mathcal R_q(w)-\mathcal R_{q,R}(w)\|_{g_d}
  =
  0.
\]
Let \(\varepsilon>0\). Choose \(R\) sufficiently large so that
\[
  \|\mathcal R_q(w)-\mathcal R_{q,R}(w)\|_{g_d}
  <
  \frac{\varepsilon}{3}
  \qquad
  \text{for every } w\in B_M.
\]
Since \(\mathcal R_{q,R}(B_M)\) is relatively compact, it is totally
bounded. Hence there exist finitely many elements
\(v_1,\dots,v_N\in\mathbb X_{g_d}\) such that, for every \(w\in B_M\),
one can choose \(j\in\{1,\dots,N\}\) for which
\[
  \|\mathcal R_{q,R}(w)-v_j\|_{g_d}
  <
  \frac{\varepsilon}{3}.
\]
For this choice of \(j\),
\[
  \|\mathcal R_q(w)-v_j\|_{g_d}
  \le
  \|\mathcal R_q(w)-\mathcal R_{q,R}(w)\|_{g_d}
  +
  \|\mathcal R_{q,R}(w)-v_j\|_{g_d}
  <
  \frac{2\varepsilon}{3}
  <
  \varepsilon.
\]
Thus the finite set \(\{v_1,\dots,v_N\}\) is an \(\varepsilon\)-net for
\(\mathcal R_q(B_M)\). Therefore \(\mathcal R_q(B_M)\) is totally bounded.
Since \(\mathbb X_{g_d}\) is a Banach space, the closure of
\(\mathcal R_q(B_M)\) is compact. Hence \(\mathcal R_q(B_M)\) is relatively
compact. Since \(M>0\) was arbitrary and every bounded subset of
\(\mathbb X_{g_d}\) is contained in some \(B_M\),
\(\mathcal R_q\) maps bounded sets into relatively compact sets.

\medskip
\noindent\textit{We next prove the continuity of \(\mathcal R_q\).}
If \(\|w\|_{g_d},\|z\|_{g_d}\le M\), then
\[
  \bigl||w(y)|^q-|z(y)|^q\bigr|
  \le
  qM^{q-1}g_d(y)^q\|w-z\|_{g_d},
\]
and hence, using \eqref{eq:g-Phi-convolution} and the definition of
\(\eta_q\),
\begin{equation}\label{eq:Rq-Lipschitz}
  \|\mathcal R_q(w)-\mathcal R_q(z)\|_{g_d}
  \le
  qM^{q-1}\eta_q\|w-z\|_{g_d}.
\end{equation}
Thus \(\mathcal R_q\) is continuous.
\end{proof}

We also recall the following standard form of the Schauder
fixed-point theorem; see \cite{Schauder}.

\begin{lemma}\label{lem:schauder-fixed-point}
Let $\bX$ be a real Banach space and let $A\subset \bX$ be nonempty,
closed, bounded, and convex. Suppose that $\cT:A\to A$ is continuous
and that $\cT(A)$ is relatively compact in $\bX$. Then there exists
$u\in A$ such that $\cT(u)=u$.
\end{lemma}

\medskip

\begin{proof}[Proof of Theorem  \ref{thm:intro-whole}]
Define
\[
  \mathcal K_q
  =
  \left\{
    w\in \mathbb X_{g_d}:
    w(0)=1,\ 0\le w\le \frac{q}{q-1}g_d
  \right\}.
\]
Clearly, \(\mathcal K_q\) is nonempty, closed, bounded, and convex.
Define
\[
  \cT_q[w]
  =
  \frac{g_d*(Q_\beta w^q)}{g_d*(Q_\beta w^q)(0)},
  \qquad
  w\in\mathcal K_q.
\]
By the definition of \(\mathcal R_q\),
\begin{equation}\label{Tq(w)}
  \cT_q[w]
  =
  \frac{g_d+\mathcal R_q(w)}
       {1+\mathcal R_q(w)(0)}.
\end{equation}
Since \(w\le \frac{q}{q-1}g_d\), by Proposition  \ref{lem:eta-q} and the
definitions of \(c_q\) and \(\eta_q\),
\[
  \|\mathcal R_q(w)\|_{g_d}
  \le
  \left(\frac{q}{q-1}\right)^q\eta_q
  \le
  \left(\frac{q}{q-1}\right)^q c_q
  =
  \frac{q}{q-1}-1.
\]
It follows that
\[
  0<\cT_q[w]
  =
  \frac{g_d+\mathcal R_q(w)}{1+\mathcal R_q(w)(0)}
  \le
  g_d+\mathcal R_q(w)
  \le
  \frac{q}{q-1}g_d.
\]
Besides, \(\cT_q[w](0)=1\), so \(\cT_q\) maps \(\mathcal K_q\) into itself.
Moreover, since \(\mathcal R_q\) is continuous and compact on bounded
subsets of \(\mathbb X_{g_d}\) by Lemma  \ref{lem:compactness-Rq}, and
\(\mathcal K_q\) is bounded, \(\cT_q\) is continuous and compact.
By the Schauder fixed point theorem, see Lemma \ref{lem:schauder-fixed-point},   there exists \(w\in\mathcal K_q\)
such that
$$\cT_q[w]=w. $$

At this fixed point,
\[
  0\le\mathcal R_q(w)(0)\le \frac{q}{q-1}-1,
\]
which leads to $\frac1{1+\mathcal R_q(w)(0)}\geq  \frac{q-1}{q}$
and then
\begin{equation}\label{eq:comparison}
  \frac{q-1}{q}g_d
  \le
  w
  \le
  \frac{q}{q-1}g_d.
\end{equation}
Let
\[
  D=g_d*(Q_\beta w^q)(0),
  \qquad
  v=D^{-1/(q-1)}w.
\]
Then
\[
  Q_\beta v^q
  =
  D^{-q/(q-1)}Q_\beta w^q,
\]
and hence
\[
  g_d*(Q_\beta v^q)
  =
  D^{-q/(q-1)}g_d*(Q_\beta w^q)
  =
  D^{-q/(q-1)}Dw
  =
  D^{-1/(q-1)}w
  =
  v.
\]
Thus
\[
  v=g_d*(Q_\beta v^q)
  =
  \frac{1}{\Phi_d(0)}\Phi_d*(Q_\beta v^q).
\]
Applying \(-\Delta\) gives
\[
  -\Delta v
  =
  \frac{1}{\Phi_d(0)}Q_\beta v^q.
\]
To obtain a solution of the original equation, set
\[
 u=\Phi_d(0)^{-1/(q-1)}v.
\]
Then
\[
  -\Delta u
  =
  \Phi_d(0)^{-1/(q-1)}\left(-\Delta v\right)
  =
  \Phi_d(0)^{-1/(q-1)}\frac{1}{\Phi_d(0)}Q_\beta v^q
  =
  Q_\beta u^q.
\]
Moreover,
\[
  u(x)\asymp v(x)\asymp g_d(x)\asymp(1+|x|)^{2-d}.
\]
Finally, \(Q_\beta u^q\in\ell^1(\mathbb Z^d)\), and since
\(\Phi_d\le\Phi_d(0)\), the sum
\[
  \sum_{y\in\mathbb Z^d}
  \left|\Phi_d(x-y)Q_\beta(y)u(y)^q\right|
  \le
  \Phi_d(0)\sum_{y\in\mathbb Z^d}Q_\beta(y)u(y)^q
  <\infty
\]
converges absolutely. Therefore
\[
  -\Delta u(x)
  =
  \sum_{y\in\mathbb Z^d}
  [-\Delta_x\Phi_d(x-y)]Q_\beta(y)u(y)^q
  =
  Q_\beta(x)u(x)^q.
\]
This proves the theorem.
\end{proof}

  \subsection{Existence in $\Z^d_+$ with $d\geq 2$}
\begin{lemma}\label{lem:half-compactness-Rq}
Let $q^\sharp$ be defined in Proposition \ref{lem:half-eta-q}, $q\ge q^\sharp$ and
\[
 \mathcal R_{q,+}(w)=\frac{1}{\Phi_{d,+}(e_1,e_1)}\Phi_{d,+}*\left(Q_\beta|w|^q\mathbf 1_{\Z_{+}^d\setminus\{e_1\}}\right)
 \qquad for\  w\in\bX_{g_{d,+}},
\]
where $e_1=(1,0,\cdots,0)$. Then $\mathcal R_{q,+}$ is continuous and compact on every bounded subset of $\bX_{g_{d,+}}$.
\end{lemma}
\begin{proof}
Put
$$s_+=q(d-1)-\beta>d+1.$$
The inequality $s_+>d+1$ is equivalent to $q>\frac{d+1+\beta}{d-1}$.

If $\|w\|_{g_{d,+}}\leq M$, then
$|w(y)|\leq Mg_{d,+}(y)$. Using  $g_{d,+}\asymp \frac{x_1}{(1+|x|)^d}$ and
$x_1\leq1+|x|$, we obtain
\begin{equation}\label{eq:half-source-bound}
 Q_\beta(y)|w(y)|^q\leq C_M(1+|y|)^{\beta+q(1-d)}=C_M(1+|y|)^{-s_+}.
\end{equation}

Choose $s'$ with $d+1<s'<s_+$. For $|y|>R$,
\[
 (1+|y|)^{-s_+}\leq
 (1+R)^{-(s_+-s')}(1+|y|)^{-s'}.
\]
Applying Lemma  \ref{lem:half-weighted-green} with $s'>d+1$, we obtain
\begin{equation}\label{eq:half-uniform-tail}
\begin{aligned}
&\quad \sup_{\|w\|_{g_{d,+}}\le M}
\bigg\|
\frac{1}{\Phi_{d,+}(e_1,e_1)}
\sum_{\substack{y\in\mathbb Z_+^d,\, |y|>R}}
\Phi_{d,+}(\cdot,y)Q_\beta(y)|w(y)|^q
\bigg\|_{g_{d,+}}
\\
& \le
\frac{C_M}{\Phi_{d,+}(e_1,e_1)}
(1+R)^{-(s_+-s')}
\sup_{x\in\mathbb Z_+^d}
\frac{
\displaystyle\sum_{y\in\mathbb Z_+^d}
\Phi_{d,+}(x,y)(1+|y|)^{-s'}
}{
g_{d,+}(x)
}
\\
& \le C_M(1+R)^{-(s_+-s')}
\rightarrow0
\qquad {\rm as}\ R\to\infty.
\end{aligned}
\end{equation}

For each fixed $y\in\Z_{+}^d$, \eqref{eq:half-green-upper} and $1+|x|\leq(1+|x-y|)(1+|y|)$,
\begin{equation}\label{eq:half-kernel-in-X}
 \frac{\Phi_{d,+}(x,y)}{g_{d,+}(x)}
 \leq Cy_1\left(\frac{1+|x|}{1+|x-y|}\right)^d
 \leq Cy_1(1+|y|)^d.
\end{equation}
Thus
\begin{equation}\label{bounded}
\|\Phi_{d,+}(\cdot,y)\|_{g_{d,+}}\le Cy_1(1+|y|)^d<\infty,
\end{equation}
so $\Phi_{d,+}(\cdot,y)\in \mathbb{X}_{g_{d,+}}.$

For $w\in \mathbb{X}_{g_{d,+}}$, let
\[
 \mathcal R_{q,+,R}(w)(x)=
 \frac{1}{\Phi_{d,+}(e_1,e_1)}
 \sum_{\substack{|y|\leq R,\ y\neq e_1}}
 \Phi_{d,+}(x,y)Q_\beta(y)|w(y)|^q.
\]
\medskip
\noindent\textit{We next prove the continuity of $\mathcal R_{q,+,R}$.}
Suppose $w_n\to w$ in $\mathbb{X}_{g_{d,+}}$. Then for each fixed $y$ with $|y|\leq R$ and $y\neq e_1$,
\[
 |w_n(y)-w(y)|\leq g_{d,+}(y)\|w_n-w\|_{g_{d,+}}\rightarrow0, {\rm as}\ n\rightarrow\infty.
\]
Since the sequence $\{w_n\}$ is bounded in $\mathbb{X}_{g_{d,+}}$, there exists $M>0$ such that
$$\|w_n\|_{g_{d,+}}\leq M\ {\rm and}\ \|w\|_{g_{d,+}}\leq M.$$
By the mean value theorem applied to $t\mapsto t^q$ on $[0,\infty)$,
\[
 \bigl||w_n(y)|^q-|w(y)|^q\bigr|
 \leq qM^{q-1}g_{d,+}(y)^q\|w_n-w\|_{g_{d,+}}.
\]
Since $\|\Phi_{d,+}(\cdot,y)\|_{g_{d,+}}<\infty$ by (\ref{bounded}) and $\{y:|y|\leq R,~y\neq e_1\}$ is finite,
\[
\begin{aligned}
&\quad\|\mathcal R_{q,+,R}(w_n)-\mathcal R_{q,+,R}(w)\|_{g_{d,+}}\\
  &\le
  \frac{1}{\Phi_{d,+}(e_1,e_1)}
  \sum_{|y|\le R,~y\neq e_1}
  Q_\beta(y)\|\Phi_{d,+}(\cdot,y)\|_{g_{d,+}}
  \bigl||w_n(y)|^q-|w(y)|^q\bigr|\\
  &\le
  \frac{qM^{q-1}}{\Phi_{d,+}(e_1,e_1)}
  \Big(
    \sum_{|y|\le R,~y\neq e_1}
    Q_\beta(y)g_{d,+}(y)^q\|\Phi_{d,+}(\cdot,y)\|_{g_{d,+}}
  \Big)
  \|w_n-w\|_{g_{d,+}}.
\end{aligned}
\]
The sum in parentheses is finite and independent of \(n\), because
\[
  Q_\beta(y)g_{d,+}(y)^q\|\Phi_{d,+}(\cdot,y)\|_{g_{d,+}}
  \le
  C(1+|y|)^{-s_+}y_1(1+|y|)^{d}
  =
  Cy_1(1+|y|)^{-s_{+}+d},
\]
and \(-s_{+}+d<0\) since \(s_+>d\). Hence
\[
  \|\mathcal R_{q,+,R}(w_n)-\mathcal R_{q,+,R}(w)\|_{g_{d,+}}
  \rightarrow0,
  \qquad {\rm as}\ n\to\infty.
\]
Thus $\mathcal R_{q,+,R}$ is continuous.

\medskip
\noindent\textit{We next show the compactness of $\mathcal R_{q,+}$.}
For every $w\in \mathbb{X}_{g_{d,+}}$,
\[
 \mathcal R_{q,+,R}(w)
 =\sum_{\substack{|y|\leq R,\ y\neq e_1}}
 \frac{Q_\beta(y)|w(y)|^q}{\Phi_{d,+}(e_1,e_1)}
 \Phi_{d,+}(\cdot,y)\in V_{R,+},
\]
where
\[
 V_{R,+}:=\operatorname{span}
 \{\Phi_{d,+}(\cdot,y):y\in\mathbb{Z}_{+}^d,\ |y|\leq R,\ y\neq e_1\}.
\]
Let $B_M=\{w:\|w\|_{g_{d,+}}\leq M\}\subset \mathbb{X}_{g_{d,+}}$ be bounded.
Then
\[
 \left|\frac{Q_\beta(y)|w(y)|^q}{\Phi_{d,+}(e_1,e_1)}\right|
 \leq\frac{M^q}{\Phi_{d,+}(e_1,e_1)}Q_\beta(y)g_{d,+}(y)^q.
\]
For fixed $R$, the set $\{y:y\in\mathbb{Z}_{+}^d,\ |y|\leq R,\ y\neq e_1\}$ is finite, so $\mathcal R_{q,+,R}(B_M)$ is bounded
in the finite-dimensional space $V_{R,+}$ and relatively compact in
$\mathbb{X}_{g_{d,+}}$. Hence $\mathcal R_{q,+,R}$ is compact.\\
By \eqref{eq:half-uniform-tail},
\[
 \lim_{R\to\infty}\sup_{w\in B_M}
 \|\mathcal R_{q,+}(w)-\mathcal R_{q,+,R}(w)\|_{g_{d,+}}=0.
\]
Let $\varepsilon>0$. Choose $R$ sufficiently large so that
$$
\|\mathcal R_{q,+}(w)-\mathcal R_{q,+,R}(w)\|_{g_{d,+}}\le\frac{\varepsilon}{3}\qquad
  \text{for every } w\in B_M.
$$
Since $\mathcal R_{q,+,R}(B_M)$ is relatively compact, there exist finitely many elements
$\overline{v}_1,\ldots,\overline{v}_N\in \mathbb{X}_{g_{d,+}}$ such that, for each $w\in B_M$, some $j$ satisfies
\[
 \|\mathcal R_{q,+,R}(w)-\overline{v}_j\|_{g_{d,+}}<\frac{\varepsilon}{3}.
\]
For this choice of $j$,
\[
\begin{aligned}
 \|\mathcal R_{q,+}(w)-\overline{v}_j\|_{g_{d,+}}
 &\leq\|\mathcal R_{q,+}(w)-\mathcal R_{q,+,R}(w)\|_{g_{d,+}}
\\[1mm]
 &\quad\  +\|\mathcal R_{q,+,R}(w)-\overline{v}_j\|_{g_{d,+}}<\frac{2\varepsilon}{3}<\varepsilon.
\end{aligned}
\]
Thus $\mathcal R_{q,+}(B_M)$ is totally bounded. Since \(\mathbb X_{g_{d,+}}\) is a Banach space, the closure of
\(\mathcal R_{q,+}(B_M)\) is compact.
Since \(M>0\) was arbitrary and every bounded subset of
\(\mathbb X_{g_{d,+}}\) is contained in some \(B_M\),
\(\mathcal R_{q,+}\) maps bounded sets into relatively compact sets.

\medskip
\noindent\textit{We prove the continuity of $\mathcal R_{q,+}$.}
If $\|w\|_{g_{d,+}},\|z\|_{g_{d,+}}\leq M$, the mean value theorem gives
\[
 \bigl||w(y)|^q-|z(y)|^q\bigr|
 \leq qM^{q-1}g_{d,+}(y)^q\|w-z\|_{g_{d,+}}.
\]
By (\ref{eq:half-uniform-tail}) and the definition of $\eta_{q,+}$,
\begin{equation}\label{eq:half-remainder-continuity}
 \|\mathcal R_{q,+}(w)-\mathcal R_{q,+}(z)\|_{g_{d,+}}
 \leq qM^{q-1}\eta_{q,+}\|w-z\|_{g_{d,+}}.
\end{equation}
Thus $\mathcal R_{q,+}$ is continuous.
\end{proof}

Next, we prove the Theorem  \ref{thm:intro-half-equality} by the Lemma \ref{lem:schauder-fixed-point}.

\begin{proof}[Proof of Theorem  \ref{thm:intro-half-equality}]
Define
\[
 \mathcal K_{q,+}=\left\{w\in \mathbb{X}_{g_{d,+}}:w(e_1)=1,\quad
 0\leq w\leq\frac{q}{q-1}g_{d,+}\right\}.
\]
Clearly, $\mathcal K_{q,+}$ is nonempty, closed, bounded, and convex.
Define
\[
 \cT_{q,+}[w]=
 \frac{  \Phi_{d,+}*(Q_\beta w^q)}
 {( \Phi_{d,+}*(Q_\beta w^q))(e_1)}
 \quad{\rm for}\ \,  w\in\mathcal K_{q,+}.
\]
By the definition of $\mathcal R_{q,+}$ and $w(e_1)=1$,
\begin{equation}\label{eq:half-normalized-map}
 \cT_{q,+}[w]=
 \frac{ g_{d,+}+\mathcal R_{q,+}(w)}
 {1+\mathcal R_{q,+}(w)(e_1)}.
\end{equation}
Since $w\leq\frac{q}{q-1}g_{d,+}$, by Proposition  \ref{lem:half-eta-q} and the definition of $\eta_{q,+}$ and $c_q$,
\[
 \|\mathcal R_{q,+}(w)\|_{g_{d,+}}
 \leq\left(\frac{q}{q-1}\right)^q\eta_{q,+}
 \leq\left(\frac{q}{q-1}\right)^qc_q
 =\frac{q}{q-1}-1.
\]
It follows that
\[
 0<\cT_{q,+}[w]
 =\frac{g_{d,+}+\mathcal R_{q,+}(w)}{1+\mathcal R_{q,+}(w)(e_1)}
 \leq g_{d,+}+\mathcal R_{q,+}(w)
 \leq\frac{q}{q-1}g_{d,+}.
\]
Besides, $\cT_{q,+}[w](e_1)=1$, so $\cT_{q,+}$ maps $\mathcal K_{q,+}$
into itself. Moreover, $\mathcal R_{q,+}$ is continuous and compact
on bounded subsets of $\mathbb{X}_{g_{d,+}}$ by Lemma  \ref{lem:half-compactness-Rq},
and $\mathcal K_{q,+}$ is bounded, $\cT_{q,+}$ is continuous and compact. By the
Schauder fixed point theorem, there exists $\overline{w}\in\mathcal K_{q,+}$ such that
$$\cT_{q,+}[\overline{w}]=\overline{w}.$$
At this fixed point,
\[
 0\leq\mathcal R_{q,+}(\overline{w})(e_1)\leq\frac{q}{q-1}-1,
\]
which leads to $\frac{1}{1+\mathcal R_{q,+}(\overline{w})(e_1)}\ge\frac{q-1}{q}$ and then
\begin{equation}\label{eq:half-comparison}
 \frac{q-1}{q}g_{d,+}\leq \overline{w}\leq\frac{q}{q-1}g_{d,+}.
\end{equation}
Let
\[
 D_+=\frac{\Phi_{d,+}*(Q_\beta \overline{w}^q)(e_{1})}{\Phi_{d,+}(e_1,e_1)},
 \qquad \overline{v}=D_+^{-1/(q-1)}\overline{w}.
\]
Then
$$Q_\beta \overline{v}^q=D_+^{-q/(q-1)}Q_\beta \overline{w}^q,$$
and
\begin{align*}
 \frac{1}{\Phi_{d,+}(e_1,e_1)}\Phi_{d,+}*(Q_\beta \overline{v}^q)
 &=\frac{1}{\Phi_{d,+}(e_1,e_1)}D_+^{-q/(q-1)}\Phi_{d,+}*(Q_\beta \overline{w}^q)
\\[1mm]& =D_+^{-1/(q-1)}\overline{w}
\\[1mm]&=\overline{v}\quad {\rm in}\ \mathbb Z_+^d.
\end{align*}
That is,
$$
\overline{v}=\frac{1}{\Phi_{d,+}(e_1,e_1)}\Phi_{d,+}*(Q_\beta \overline{v}^q) \quad {\rm in}\ \mathbb Z_+^d,
$$
which implies that
\[
 -\Delta \overline{v}=\frac{1}{\Phi_{d,+}(e_1,e_1)}Q_\beta \overline{v}^q,\qquad \overline{v}=0\quad {\rm in}\ \Z^d\setminus \Z^d_+.
\]
To obtain a solution of the original equation, set
$$\overline{u}=\Phi_{d,+}(e_1,e_1)^{-1/(q-1)}\overline{v}\quad {\rm in}\ \mathbb Z_+^d.$$
Then
\[
-\Delta \overline{u}=Q_\beta \overline{u}^q\quad {\rm in}\ \Z^d_+,\qquad \overline{u}=0\quad {\rm in}\ \Z^d\setminus \Z^d_+.
\]
Moreover
\[
 \overline{u}(x)\asymp g_{d,+}(x)\asymp\frac{x_1}{(1+|x|)^d}.
\]
By \eqref{eq:half-comparison} and the scaling defined above,
\[
0<\overline{u}(y)
\leq Cg_{d,+}(y)
\leq C\frac{y_1}{(1+|y|)^d}
\leq C(1+|y|)^{1-d}.
\]
Here we used $y_1\leq 1+|y|$. Moreover, assumption
$(\mathcal Q_0)$ with $x_0=e_1$ gives
\[
0\leq Q_\beta(y)
\leq (1+|y-e_1|)^\beta
\leq C(1+|y|)^\beta.
\]
Consequently,
\[
\begin{aligned}
0\leq Q_\beta(y)\overline{u}(y)^q\leq C(1+|y|)^{\beta+q(1-d)}
=C(1+|y|)^{-s_+},
\end{aligned}
\]
where
\[
s_+=q(d-1)-\beta>d+1.
\]
In particular, $Q_\beta(y)\overline{u}(y)^q\in\ell^1(\mathbb Z_+^d)$. By Lemma  \ref{lem:half-weighted-green}, for every fixed
$z\in\mathbb Z_+^d$,
\[
\begin{aligned}
\sum_{y\in\mathbb Z_+^d}
\left|\Phi_{d,+}(z,y)Q_\beta(y)\overline{u}(y)^q\right|
&\leq C\sum_{y\in\mathbb Z_+^d}
\Phi_{d,+}(z,y)(1+|y|)^{-s_+}\leq Cg_{d,+}(z)<\infty.
\end{aligned}
\]
Therefore
\[
 -\Delta\overline{u}(x)
 =\sum_{y\in\Z_{+}^d}[-\Delta_{x}\Phi_{d,+}(x,y)]
 Q_\beta(y)\overline{u}(y)^q
 =Q_\beta(x)\overline{u}(x)^q.
\]
This proves the theorem.
\end{proof}

\appendix

\section{The two-dimensional Dirichlet Green kernel}
\label{app:two-dimensional-kernel}

This appendix constructs the two-dimensional Dirichlet Green
kernel used in  $\Z^2_+$ under the zero Dirichlet boundary condition.
Since the simple random walk on $\mathbb Z^2$ is recurrent, the
construction is based on the two-dimensional potential kernel
rather than on a positive whole-space Green function.

\begin{lemma}\label{lem:two-potential-kernel}
For $(m,n)\in\mathbb Z^2$, define
\begin{equation}\label{eq:two-potential-definition}
 A(m,n)=\frac{1}{4\pi^2}
 \int_{-\pi}^{\pi}\int_{-\pi}^{\pi}
 \frac{1-\cos(m\theta+n\varphi)}
 {4-2\cos\theta-2\cos\varphi}
 \,d\theta\,d\varphi.
\end{equation}
Then $A(m,n)$ is finite,
\[
 A(-m,n)=A(m,n)=A(m,-n),
 \qquad A(m,n)=A(n,m),
\]
and for the unnormalized lattice Laplacian,
\begin{equation}\label{eq:two-potential-equation}
 \Delta A=\delta_0\qquad\text{on }\mathbb Z^2.
\end{equation}
\end{lemma}

\begin{proof}
Put
\[
 \lambda(\theta,\varphi)
 =4-2\cos\theta-2\cos\varphi.
\]
We first compare it with $\theta^2+\varphi^2$.
For $0\leq s\le\frac{\pi}{2}$, define $H(s)=s\cos s-\sin s$. Then
\[
 H(0)=0,\qquad H'(s)=-s\sin s\leq 0.
\]
Hence $H(s)\leq0$ and for $0\leq s\le\frac{\pi}{2}$,
\[
 \left(\frac{\sin s}{s}\right)'=\frac{s\cos s-\sin s}{s^2}\leq0.
\]
Therefore
\[
 \frac{\sin s}{s}\geq \frac{\sin(\frac{\pi}{2})}{\frac{\pi}{2}}=\frac{2}{\pi}.
\]
Besides, using
\[
1-\cos t=2\sin^2\left(\frac{t}{2}\right),
\qquad
|\sin s|\leq |s|,
\]
together with the preceding lower bound for $\frac{\sin s}{s}$, we obtain
\begin{equation}\label{eq:(2.2)}
\begin{aligned}
0\leq 1-\cos t&\leq\frac{t^2}{2},
&& t\in\mathbb R,\\
1-\cos t&\geq\frac{2t^2}{\pi^2},
&& |t|\leq\pi.
\end{aligned}
\end{equation}

Now write
\[
 \lambda(\theta,\varphi)=2(1-\cos\theta)+2(1-\cos\varphi).
\]
Applying \eqref{eq:(2.2)} to each term, we have
\begin{equation}\label{eq: estimimate}
 \frac{4}{\pi^2}(\theta^2+\varphi^2)
 \leq\lambda(\theta,\varphi)\leq\theta^2+\varphi^2,
\end{equation}
and
\[
 1-\cos(m\theta+n\varphi)\leq\frac12(m\theta+n\varphi)^2.
\]
Moreover,
\[
 (m^2+n^2)(\theta^2+\varphi^2)-(m\theta+n\varphi)^2=(m\varphi-n\theta)^2\geq0,
\]
so we have
\[
(m\theta+n\varphi)^2\le(m^2+n^2)(\theta^2+\varphi^2).
\]
Combining these two inequalities with \eqref{eq: estimimate}, we get
\[
 0\leq\frac{1-\cos(m\theta+n\varphi)}{\lambda(\theta,\varphi)}
 \leq\frac{\pi^2}{8}(m^2+n^2).
\]
For fixed $m,n$, the right-hand side is a finite constant.

Expand the numerator we have
\[
 1-\cos(m\theta+n\varphi)
 =1-\cos(m\theta)\cos(n\varphi)+\sin(m\theta)\sin(n\varphi).
\]
The last term divided by $\lambda$ is absolutely integrable, because
\[
 \frac{|\sin(m\theta)\sin(n\varphi)|}{\lambda(\theta,\varphi)}
 \leq\frac{\pi^2|mn|}{4}\frac{|\theta\varphi|}{\theta^2+\varphi^2}
 \leq\frac{\pi^2|mn|}{8}.
\]
Here we used $2|\theta\varphi|\leq\theta^2+\varphi^2$. Note that $\sin(m\theta)$ and $\sin(n\varphi)$ is odd. Consequently, the integral of the last term is zero, and
\begin{equation}
 A(m,n)=\frac{1}{4\pi^2}\int_{-\pi}^{\pi}\int_{-\pi}^{\pi}
 \frac{1-\cos(m\theta)\cos(n\varphi)}{\lambda(\theta,\varphi)}\dd\theta\dd\varphi.
\end{equation}
this formula gives
\begin{equation}
 A(-m,n)=A(m,n)=A(m,-n),\qquad A(0,0)=0.
\end{equation}
Interchanging the two integration variables also gives $A(m,n)=A(n,m)$.

For fixed $\theta,\varphi$, put $\psi=m\theta+n\varphi$. The addition formula for cosine gives
\begin{align*}
 \cos(\psi+\theta)+\cos(\psi-\theta)&=2\cos\theta\cos\psi,\\
 \cos(\psi+\varphi)+\cos(\psi-\varphi)&=2\cos\varphi\cos\psi.
\end{align*}
It follows from the definition of $\Delta$ that
\[
 \Delta_{m,n}\cos\psi=(2\cos\theta+2\cos\varphi-4)\cos\psi
 =-\lambda(\theta,\varphi)\cos\psi.
\]
Since the Laplacian of a constant is zero, we have
\[
 \Delta_{m,n}(1-\cos\psi)=\lambda(\theta,\varphi)\cos\psi.
\]
The operator $\Delta$ is a finite linear combination of five function values. Each corresponding integral
in \eqref{eq:two-potential-definition} is absolutely convergent. We may therefore move this finite combination inside the
integral. The denominator is independent of $m,n$ and cancels:
\begin{align}\label{eq: A}
 \Delta A(m,n)
 &=\frac{1}{4\pi^2}\int_{-\pi}^{\pi}\int_{-\pi}^{\pi}
     \cos(m\theta+n\varphi)\dd\theta\dd\varphi\notag\\
 &=\left(\frac{1}{2\pi}\int_{-\pi}^{\pi}\cos(m\theta)\dd\theta\right)
   \left(\frac{1}{2\pi}\int_{-\pi}^{\pi}\cos(n\varphi)\dd\varphi\right).
\end{align}
Since
\[
 \int_{-\pi}^{\pi}\cos(jt)\dd t=
 \begin{cases}
 2\pi,&j=0,\\
 0,&j\neq0.
 \end{cases}
\]
So we have
\begin{equation}
 \Delta A=\delta_0.
\end{equation}
\end{proof}

We next record the estimate that provides positivity after
reflection.

\begin{lemma}\label{lem:two-short-difference}
For $h\in\mathbb{N}$, $h\geq1$ and $k\in\mathbb Z$, let
\[
 P_h(k)=A(h+1,k)-A(h-1,k).
\]
There are constants $c_P,C_P>0$, independent of $h$ and $k$, such
that
\begin{equation}\label{eq:two-short-estimate}
 c_P\frac{h}{h^2+k^2}
 \leq P_h(k)\leq
 C_P\frac{h}{h^2+k^2}.
\end{equation}
In particular, $P_h(k)>0$.
\end{lemma}

\begin{proof}
For $-\pi\leq t\leq\pi$, let $\alpha(t)\geq0$ be defined by
\[
\cosh\alpha(t)=2-\cos t.
\]
Then $\alpha$ is even and
\begin{align*}
P_h(k)
&=A(h+1,k)-A(h-1,k)\\
&=\frac{1}{4\pi^2}
  \int_{-\pi}^{\pi}\int_{-\pi}^{\pi}
  \frac{\cos\bigl((h-1)\theta+k\varphi\bigr)
        -\cos\bigl((h+1)\theta+k\varphi\bigr)}
       {4-2\cos\theta-2\cos\varphi}
  \,d\theta\,d\varphi\\
&=\frac{1}{4\pi^2}
  \int_{-\pi}^{\pi}\int_{-\pi}^{\pi}
  \frac{2\sin(h\theta+k\varphi)\sin\theta}
       {4-2\cos\theta-2\cos\varphi}
  \,d\theta\,d\varphi,
\end{align*}
where we used
\[
\cos(u-v)-\cos(u+v)=2\sin u\sin v.
\]
Since
\[
\cosh \alpha(\varphi)=2-\cos\varphi,
\qquad
\alpha(\varphi)\geq 0,
\]
we have
\[
4-2\cos\theta-2\cos\varphi
=2\bigl(\cosh \alpha(\varphi)-\cos\theta\bigr),
\]
and
\begin{align*}
P_h(k)
&=\frac{1}{4\pi^2}
  \int_{-\pi}^{\pi}\int_{-\pi}^{\pi}
  \frac{\sin(h\theta+k\varphi)\sin\theta}
       {\cosh \alpha(\varphi)-\cos\theta}
  \,d\theta\,d\varphi\\[1mm]
&=\frac{1}{4\pi^2}
  \int_{-\pi}^{\pi}\int_{-\pi}^{\pi}
  \frac{\sin(h\theta)\cos(k\varphi)\sin\theta
        +\sin(k\varphi)\cos(h\theta)\sin\theta}
       {\cosh \alpha(\varphi)-\cos\theta}
  \,d\theta\,d\varphi.
\end{align*}
For every fixed $\varphi$, the function
\[
\theta\longmapsto
\frac{\sin(k\varphi)\cos(h\theta)\sin\theta}
     {\cosh \alpha(\varphi)-\cos\theta}
\]
is odd.  Hence
\begin{equation}\label{tag{1}}
P_h(k)=\frac{1}{4\pi^2}
\int_{-\pi}^{\pi}\cos(k\varphi)
\int_{-\pi}^{\pi}
\frac{\sin(h\theta)\sin\theta}
     {\cosh \alpha(\varphi)-\cos\theta}
\,d\theta\,d\varphi.
\end{equation}
For \(\varphi\neq0\), put
\[
\alpha=\alpha(\varphi)>0,
\qquad r=e^{-\alpha}\in(0,1).
\]
Then
\[
\sum_{\ell=1}^{\infty}r^\ell e^{i\ell\theta}
=\frac{re^{i\theta}}{1-re^{i\theta}}.
\]
Since the exceptional point \(\varphi=0\) has measure zero, the
resulting identity is sufficient for the outer integral.\\
Because $e^{i\ell\theta}=\cos(\ell\theta)+i\sin(\ell\theta)$, then
\begin{align*}
\operatorname{Re}\Big(
  \sum_{\ell=1}^{\infty}r^\ell e^{i\ell\theta}
\Big)
&=\sum_{\ell=1}^{\infty}r^\ell\cos(\ell\theta)
 =\operatorname{Re}\Big(
  \frac{re^{i\theta}}{1-re^{i\theta}}
\Big) =\frac{r\cos\theta-r^2}{1-2r\cos\theta+r^2}.
\end{align*}
Consequently,
\begin{equation}\label{tag{2}}
1+2\sum_{\ell=1}^{\infty}r^\ell\cos(\ell\theta)
=1+2\frac{r\cos\theta-r^2}
               {1-2r\cos\theta+r^2}=\frac{1-r^2}{1-2r\cos\theta+r^2}.
\end{equation}
On the other hand,
\begin{align*}
\cosh \alpha-\cos\theta
&=\frac{e^\alpha+e^{-\alpha}}{2}-\cos\theta
 =\frac{r^{-1}+r-2\cos\theta}{2}
 =\frac{1-2r\cos\theta+r^2}{2r},\\
\sinh \alpha
&=\frac{e^\alpha-e^{-\alpha}}{2}
 =\frac{r^{-1}-r}{2}
 =\frac{1-r^2}{2r}.
\end{align*}
Combining these identities with \eqref{tag{2}}, we obtain
\begin{equation}
\frac{\sinh \alpha}{\cosh \alpha-\cos\theta}
=1+2\sum_{\ell=1}^{\infty}e^{-\ell \alpha}\cos(\ell\theta),
\end{equation}
equivalently
\begin{equation}\label{tag{3}}
\frac{1}{\cosh \alpha-\cos\theta}
=\frac{1}{\sinh \alpha}
 \Big[1+2\sum_{\ell=1}^{\infty}
 e^{-\ell \alpha}\cos(\ell\theta)\Big].
\end{equation}
Now define
\[
I_h(\alpha)=\int_{-\pi}^{\pi}
\frac{\sin(h\theta)\sin\theta}
     {\cosh \alpha-\cos\theta}\,d\theta.
\]
Using
\[
\sin(h\theta)\sin\theta
=\frac{1}{2}\left[
\cos\bigl((h-1)\theta\bigr)
-\cos\bigl((h+1)\theta\bigr)
\right],
\]
we have
\begin{align*}
I_h(\alpha)
&=\frac{1}{2}\int_{-\pi}^{\pi}
\frac{\cos\bigl((h-1)\theta\bigr)
      -\cos\bigl((h+1)\theta\bigr)}
     {\cosh \alpha-\cos\theta}\,d\theta\\
&=\frac{1}{2}\int_{-\pi}^{\pi}\frac{\cos\bigl((h-1)\theta\bigr)}{\sinh \alpha}[1+2\sum_{\ell=1}^{\infty}
 e^{-\ell \alpha}\cos(\ell\theta)]\,d\theta\\
&~-\frac{1}{2}\int_{-\pi}^{\pi}\frac{\cos\bigl((h+1)\theta\bigr)}{\sinh \alpha}[1+2\sum_{\ell=1}^{\infty}
 e^{-\ell \alpha}\cos(\ell\theta)]\,d\theta.
\end{align*}
Moreover,
\[
\int_{-\pi}^{\pi}\cos(j\theta)\,d\theta=0
\quad (j\geq 1),
\]
and
\[
\int_{-\pi}^{\pi}\cos(j\theta)\cos(\ell\theta)\,d\theta
=\begin{cases}
\pi, & j=\ell\geq 1,\\
0,   & j\neq \ell.
\end{cases}
\]
For $j=0$, one also has
\(\int_{-\pi}^{\pi}1\,d\theta=2\pi\).
For every $h\geq 1$, we have
\begin{equation}\label{tag{4}}
I_h(\alpha)=\frac{\pi}{\sinh \alpha}\left[e^{-(h-1)\alpha}-e^{-(h+1)\alpha}\right]
=\frac{\pi}{\sinh \alpha}e^{-h\alpha}(e^\alpha-e^{-\alpha})
=2\pi e^{-h\alpha}.
\end{equation}
Here, when $h=1$, the term containing $\cos(0\theta)$ is supplied by the
constant term in \eqref{tag{3}}.
We have
\begin{equation}\label{eq:two-p-integral}
P_h(k)
=\frac{1}{4\pi^2}\int_{-\pi}^{\pi}
\cos(k\varphi)\,2\pi e^{-h \alpha(\varphi)}\,d\varphi
=\frac{1}{2\pi}\int_{-\pi}^{\pi}
e^{-h \alpha(\varphi)}\cos(k\varphi)\,d\varphi.
\end{equation}
Since $\alpha(\varphi)=\alpha(-\varphi)$ and
$\cos(k\varphi)=\cos(-k\varphi)$, the integrand is even. Besides, by
\[
\cosh\alpha(t)=2-\cos t=1+2\sin^2(\frac{t}{2})
\]
and
\[
\cosh \alpha(t)=1+2\sinh^2(\frac{\alpha(t)}{2}),
\]
we have
\[
\sin^2(\frac{t}{2})=\sinh^2(\frac{\alpha(t)}{2}).
\]
Since $0\leq t\leq\pi$, then we have
\[
\frac{\alpha(t)}{2}=\operatorname{arsinh}(\sin(\frac{t}{2})).
\]
Equivalently,
\begin{equation}\label{eq:two-alpha}
 \alpha(t)=2\operatorname{arsinh}\!\left(\sin\frac t2\right).
\end{equation}

Differentiating \eqref{eq:two-alpha} gives
\[
 \alpha'(t)=\frac{\cos(t/2)}
 {\sqrt{1+\sin^2(t/2)}},
 \qquad
 \alpha''(t)=-\frac{\sin(t/2)}
 {(1+\sin^2(t/2))^{3/2}}.
\]
In particular,
\[
 0\leq\alpha'\leq1,\qquad \alpha''\leq0,\qquad
 \alpha'(0)=1,\qquad \alpha'(\pi)=0.
\]
Moreover, for $0\leq s\leq1$, we have
\[
\operatorname{arsinh}s
=\int_0^s\frac{1}{\sqrt{1+\tau^2}}\,d\tau
\geq\frac{s}{\sqrt{2}}.
\]
Moreover,
\[
\sin\frac{t}{2}\geq\frac{t}{\pi},
\qquad 0\leq t\leq\pi.
\]
Therefore,
\[
\alpha(t)
=2\operatorname{arsinh}\left(\sin\frac{t}{2}\right)
\geq\sqrt{2}\sin\frac{t}{2}
\geq\frac{\sqrt{2}}{\pi}t.
\]
On the other hand, since
\[
\operatorname{arsinh}s\leq s
\qquad\text{for }s\geq0
\]
and
\[
\sin\frac{t}{2}\leq\frac{t}{2},
\]
we obtain
\[
\alpha(t)
=2\operatorname{arsinh}\left(\sin\frac{t}{2}\right)
\leq2\sin\frac{t}{2}
\leq t.
\]
Consequently, for $c_0=\frac{\sqrt2}{\pi}$, we have
\begin{equation}\label{eq:two-alpha-comparison}
 c_0t\leq\alpha(t)\leq t,
 \qquad 0\leq t\leq\pi.
\end{equation}
Set $F_h(t)=e^{-h\alpha(t)}$.  Then
\begin{equation}\label{eq:two-f-second}
 F_h''(t)=e^{-h\alpha(t)}
 \bigl(h^2\alpha'(t)^2-h\alpha''(t)\bigr)\geq0,
\end{equation}
and
\[
 F_h'(0)=-h,\qquad F_h'(\pi)=0,\qquad
 \int_0^\pi F_h''(t)\,dt=h.
\]
For $k\neq0$, two integrations by parts in
\eqref{eq:two-p-integral}, including the boundary term at $t=0$,
give
\begin{equation}\label{eq:two-p-positive}
 P_h(k)=\frac{1}{\pi k^2}
 \int_0^\pi(1-\cos(kt))F_h''(t)\,dt.
\end{equation}
Thus $P_h(k)\geq0$.
We now prove the upper bound in \eqref{eq:two-short-estimate}. By \eqref{eq:two-p-integral},
the nonnegativity of \(P_h(k)\), and the estimate
\(\alpha(t)\geq c_0t\) in \eqref{eq:two-alpha-comparison}, we have
\begin{align*}
0\leq P_h(k)
&\leq \frac{1}{\pi}
\int_0^\pi F_h(t)\lvert\cos(kt)\rvert\,dt \\
&\qquad\ \  \leq \frac{1}{\pi}
\int_0^\pi e^{-h \alpha(t)}\,dt \,    \leq \frac{1}{\pi}
\int_0^\pi e^{-c_0ht}\,dt \leq \frac{C}{h}.
\end{align*}
Thus
\begin{equation}\label{tag{A.20}}
P_h(k)\leq \frac{C}{h}
\qquad\text{for every }k\in\mathbb Z.
\end{equation}

For \(k\neq0\), formula \eqref{eq:two-p-positive}, together with
\(0\leq 1-\cos(kt)\leq2\) and \(F_h''(t)\geq0\), gives
\begin{align*}
P_h(k)
&=\frac{1}{\pi k^2}
\int_0^\pi
\bigl(1-\cos(kt)\bigr)F_h''(t)\,dt  \leq \frac{2}{\pi k^2}
\int_0^\pi F_h''(t)\,dt.
\end{align*}
Since
\[
F_h'(0)=-h,
\qquad
F_h'(\pi)=0,
\]
the fundamental theorem of calculus yields
\[
\int_0^\pi F_h''(t)\,dt
=F_h'(\pi)-F_h'(0)
=h.
\]
Consequently,
\begin{equation}\label{tag{A.21}}
P_h(k)\leq \frac{2h}{\pi k^2},
\qquad k\neq0.
\end{equation}

Combining \eqref{tag{A.20}} and \eqref{tag{A.21}}, we obtain
\[
P_h(k)
\leq C\min\left\{\frac{1}{h},\frac{h}{k^2}\right\},
\qquad k\neq0.
\]
If \(0<|k|\leq h\), then \(h^2+k^2\leq2h^2\), and hence
\[
\frac{1}{h}\leq\frac{2h}{h^2+k^2}.
\]
If \(|k|>h\), then \(h^2+k^2\leq2k^2\), and therefore
\[
\frac{h}{k^2}\leq\frac{2h}{h^2+k^2}.
\]
For \(k=0\), estimate \eqref{tag{A.20}} directly gives
\[
P_h(0)\leq\frac{C}{h}
=C\frac{h}{h^2+0^2}.
\]
After enlarging \(C\) if necessary, we conclude that
\[
P_h(k)\leq C\frac{h}{h^2+k^2},
\qquad h\geq1,\quad k\in\mathbb Z,
\]
which is the upper bound in \eqref{eq:two-short-estimate}.

For completeness, we prove the lower bound. Recall that
\[
F_h(t)=e^{-h\alpha(t)},
\qquad
\alpha(t)=\operatorname{arcosh}(2-\cos t).
\]
Since
\[
\alpha(t)
=2\operatorname{arsinh}\left(\sin\frac{t}{2}\right),
\]
we have
\[
\alpha'(t)
=
\frac{\cos(t/2)}
{\sqrt{1+\sin^2(t/2)}}.
\]
In particular, \(0\leq\alpha'(t)\leq1\). Since \(\alpha(0)=0\),
\[
\alpha(t)=\int_0^t\alpha'(s)\,ds\leq t.
\]
Consequently, if \(0\leq t\leq 1/h\), then
\[
h\alpha(t)\leq ht\leq1,
\]
and hence
\[
e^{-h\alpha(t)}\geq e^{-1}.
\]

Moreover, since \(h\geq1\) and \(0\leq t\leq1/h\), we have
\(0\leq t\leq1\), and therefore
\[
\sin^2\frac{t}{2}
\leq\frac{t^2}{4}
\leq\frac14.
\]
It follows that
\begin{align*}
\bigl(\alpha'(t)\bigr)^2
&=
\frac{\cos^2(t/2)}
     {1+\sin^2(t/2)}
 =
\frac{1-\sin^2(t/2)}
     {1+\sin^2(t/2)}  \geq
\frac{1-\frac14}{1+\frac14}
=\frac35.
\end{align*}

By \eqref{eq:two-f-second},
\[
F_h''(t)
=
e^{-h\alpha(t)}
\left[
h^2\bigl(\alpha'(t)\bigr)^2-h\alpha''(t)
\right].
\]
Since \(\alpha''(t)\leq0\), we obtain
\[
F_h''(t)
\geq
h^2e^{-h\alpha(t)}
\bigl(\alpha'(t)\bigr)^2.
\]
Thus
\begin{equation}\label{tag{A.22}}
F_h''(t)\geq c_*h^2,
\qquad
0\leq t\leq\frac1h,
\end{equation}
where $c_*=\frac{3}{5e}.$

We now divide the proof into three cases.

\medskip
\noindent
\textit{Case 1: \(k=0\).}
By \eqref{eq:two-p-integral} and the estimate
\(e^{-h\alpha(t)}\geq e^{-1}\) on \([0,\frac{1}{h}]\), we have
\begin{align*}
P_h(0)
&=\frac{1}{\pi}\int_0^\pi F_h(t)\,dt  \geq\frac{1}{\pi}\int_0^{1/h}
e^{-h\alpha(t)}\,dt  \geq\frac{e^{-1}}{\pi}
\int_0^{1/h}dt
=\frac{e^{-1}}{\pi h}.
\end{align*}
Since
\[
\frac{h}{h^2+0^2}=\frac1h,
\]
it follows that
\[
P_h(0)
\geq\frac{e^{-1}}{\pi}
\frac{h}{h^2+0^2}.
\]

\medskip
\noindent
\textit{Case 2: \(0<|k|\leq2h\).}
For \(0\leq t\leq\frac{1}{h}\), we have
\[
|kt|\leq\frac{|k|}{h}\leq2<\pi.
\]
For every \(|x|\leq\pi\),
\[
1-\cos x
=2\sin^2\frac{x}{2}
\geq\frac{2x^2}{\pi^2}.
\]
Taking \(x=kt\), we obtain
\[
1-\cos(kt)
\geq\frac{2k^2t^2}{\pi^2},
\qquad
0\leq t\leq\frac1h.
\]
Using \eqref{eq:two-p-positive} and \eqref{tag{A.22}}, we find
\begin{align*}
P_h(k)
&\geq
\frac{1}{\pi k^2}
\int_0^{1/h}
\bigl(1-\cos(kt)\bigr)F_h''(t)\,dt \\
&\geq
\frac{1}{\pi k^2}
\int_0^{1/h}
\frac{2k^2t^2}{\pi^2}\,c_*h^2\,dt \\
&=
\frac{2c_*h^2}{\pi^3}
\int_0^{1/h}t^2\,dt \\
&=
\frac{2c_*h^2}{\pi^3}
\cdot\frac{1}{3h^3}
=\frac{2c_*}{3\pi^3h}.
\end{align*}
Since
\[
\frac{h}{h^2+k^2}\leq\frac1h,
\]
we conclude that
\[
P_h(k)
\geq
\frac{2c_*}{3\pi^3}
\frac{h}{h^2+k^2}.
\]

\medskip
\noindent
\textit{Case 3: \(|k|>2h\).}
A direct calculation gives
\begin{align*}
\int_0^{1/h}\bigl(1-\cos(kt)\bigr)\,dt
&=
\left[t-\frac{\sin(kt)}{k}\right]_0^{1/h}
 =
\frac1h-\frac{\sin(\frac{k}{h})}{k}.
\end{align*}
Since
\[
\Big|\frac{\sin(\frac{k}{h})}{k}\Big|
\leq\frac1{|k|}
\]
and \(|k|>2h\), it follows that
\begin{align*}
\int_0^{1/h}\bigl(1-\cos(kt)\bigr)\,dt
&\geq
\frac1h-\frac1{|k|} \geq
\frac1h-\frac1{2h}
=\frac1{2h}.
\end{align*}
Therefore, by \eqref{eq:two-p-positive} and \eqref{tag{A.22}},
\begin{align*}
P_h(k)
&\geq
\frac{1}{\pi k^2}
\int_0^{1/h}
\bigl(1-\cos(kt)\bigr)F_h''(t)\,dt \\
&\geq
\frac{c_*h^2}{\pi k^2}
\int_0^{1/h}
\bigl(1-\cos(kt)\bigr)\,dt \\
&\geq
\frac{c_*h^2}{\pi k^2}\cdot\frac1{2h}
=\frac{c_*h}{2\pi k^2}.
\end{align*}
Since \(h^2+k^2\geq k^2\), we have
\[
\frac1{k^2}\geq\frac1{h^2+k^2}.
\]
Consequently,
\[
P_h(k)
\geq
\frac{c_*}{2\pi}
\frac{h}{h^2+k^2}.
\]

Finally, set
\[
c=
\min\left\{
\frac{e^{-1}}{\pi},
\frac{2c_*}{3\pi^3},
\frac{c_*}{2\pi}
\right\}>0.
\]
The preceding three cases show that
\[
P_h(k)\geq c\frac{h}{h^2+k^2},
\qquad
h\geq1,\quad k\in\mathbb Z.
\]
This proves the lower bound in \eqref{eq:two-short-estimate} and completes the
proof.
\end{proof}

\begin{proof}[Proof of Proposition \ref{prop:two-dimensional-green}]
Fix \(y=(y_1,y_2)\in \mathbb{Z}_+^2\) and define
\[
\Phi_{2,+}(x,y)
   =A(x-y^{*})-A(x-y),
   \qquad x\in\mathbb{Z}_+^2.
\]
If \(x=(x_1,x_{2})=(0,x_{2})\), then
\[
x-y^{*}=(y_{1},x_{2}-y_{2}),
\qquad
x-y=(-y_{1},x_{2}-y_{2}).
\]
By the reflection symmetry of the potential kernel stated in
Lemma~A.1,
\[
A(y_{1},x_{2}-y_{2})
   =A(-y_{1},x_{2}-y_{2}).
\]
Consequently,
\[
\Phi_{2,+}(x,y)=0
\qquad\text{whenever }x_{1}=0.
\]
If \(x_1\ge2\), all neighbors of \(x\) belong to \(\mathbb{Z}_+^2\), so the
full-space and Dirichlet Laplacians agree. If \(x_1=1\), the only
neighbor of \(x\) outside \(\mathbb{Z}_+^2\) is \(x-e_1=(0,x_2)\). At this point
\(\Phi_{2,+}(x-e_1,y)=0\), which agrees with the zero extension.
Thus the restriction of \(\Phi_{2,+}(x,y)\) to \(\mathbb{Z}_+^2\), extended by zero
outside \(\mathbb{Z}_+^2\), satisfies the Dirichlet boundary condition.

Besides, by \eqref{eq:two-potential-equation}, we have
\[
-\Delta_x A(x-y^{*})=-\delta_{y^{*}}(x)
\]
and
\[
-\Delta_x\bigl(-A(x-y)\bigr)=\delta_y(x).
\]
Therefore,
\[
-\Delta_x\bigl[A(x-y^*)-A(x-y)\bigr]
=\delta_y(x)-\delta_{y^*}(x).
\]
Hence
\[
-\Delta_{+}\Phi_{2,+}(x,y)
   =\delta_y(x),
\qquad x,y\in \mathbb{Z}_+^2.
\]

We next verify symmetry. Write
\[
x=(m,x_{2}),\qquad y=(n,y_{2}).
\]
Using the invariance of \(A\) under reflections in the coordinate axes,
we obtain
\[
\begin{aligned}
A(x-y^{*})
 &=A(m+n,x_{2}-y_{2})  \\
 &=A(m+n,y_{2}-x_{2})
 =A(y-x^{*}),
\end{aligned}
\]
and similarly,
\[
A(x-y)
 =A(m-n,x_{2}-y_{2})
 =A(n-m,y_{2}-x_{2})
 =A(y-x).
\]
It follows that
\[
\Phi_{2,+}(x,y)=\Phi_{2,+}(y,x).
\]

We now derive the two-sided estimate. Put
\[
k=x_{2}-y_{2},\qquad
b=|m-n|,\qquad
N=\min\{m,n\}.
\]
Since \(m,n\geq 1\), we have \(N\geq1\) and
\[
m+n=|m-n|+2\min\{m,n\}=b+2N.
\]
By the symmetry of \(A\) in the first coordinate,
\[
A(x-y)=A(b,k),
\qquad
A(x-y^{*})=A(m+n,k).
\]
Therefore,
\[
\Phi_{2,+}(x,y)=A(m+n,k)-A(b,k).
\]

Recall that
\[
P_h(k)=A(h+1,k)-A(h-1,k),
\qquad h\geq1.
\]
Since \(m+n=b+2N\), $\Phi_{2,+}(x,y)$ can be written as
\[
\begin{aligned}
\Phi_{2,+}(x,y)&=A(m+n,k)-A(b,k)\\
 &=A(b+2N,k)-A(b,k)\\
 &=\sum_{j=0}^{N-1}
   \bigl[A(b+2j+2,k)-A(b+2j,k)\bigr] =\sum_{j=0}^{N-1}P_{b+2j+1}(k).
\end{aligned}
\]
Thus
\begin{equation}\label{tag{P double}}
\Phi_{2,+}(x,y)
   =\sum_{j=0}^{N-1}P_{b+2j+1}(k).
\end{equation}
For fixed \(x,y\in \mathbb{Z}_+^2\), the integer \(N\) is finite. Hence the
sum contains only finitely many terms, and \(\Phi_{2,+}(x,y)\) is finite. Moreover, Lemma~A.2 gives
\[
P_h(k)>0
\qquad\text{for every }h\geq1,\ k\in\mathbb Z.
\]
Since \(N\geq1\), formula \eqref{tag{P double}} implies
\[
\Phi_{2,+}(x,y)>0.
\]

Lemma~A.2 also gives constants \(c_{0},C_{0}>0\), independent of
\(h\) and \(k\), such that
\[
c_{0}\frac{h}{h^{2}+k^{2}}
 \leq P_h(k)
 \leq C_{0}\frac{h}{h^{2}+k^{2}}.
\]
Applying this estimate term by term in \eqref{tag{P double}}, we find
\begin{equation}\label{A29}
\Phi_{2,+}(x,y)
 \asymp
 \sum_{j=0}^{N-1}
 \frac{b+2j+1}{(b+2j+1)^{2}+k^{2}}.
\end{equation}

To estimate the sum, for \(h\geq1\) define
\[
J_h
 =\int_{h-1}^{h+1}
     \frac{t}{1+k^{2}+t^{2}}\,dt.
\]
By definition, it suffices to compare the denominator on the interval
\([h-1,h+1]\). Since \(t\geq h-1\geq0\), we have
\[
1+(h-1)^2-\frac{h^2}{2}
=\frac{(h-2)^2}{2}\geq0.
\]
and
\[
1+t^2\geq 1+(h-1)^2\geq\frac{h^2}{2},
\]
Consequently,
\[
1+k^2+t^2\geq\frac{h^2+k^2}{2}.
\]
It follows that
\[
\begin{aligned}
J_h
&\leq \frac{2}{h^2+k^2}
      \int_{h-1}^{h+1}t\,dt  =\frac{2}{h^2+k^2}\cdot 2h
 =\frac{4h}{h^2+k^2}.
\end{aligned}
\]
Thus,
\[
\frac{J_h}{4}\leq\frac{h}{h^2+k^2}.
\]
On the other hand, since \(t\leq h+1\), we obtain
\[
\begin{aligned}
1+k^2+t^2
&\leq 1+k^2+(h+1)^2 \\
&=k^2+h^2+2h+2  \leq 5(k^2+h^2).
\end{aligned}
\]
Indeed, because \(h\geq1\) and \(k^2\geq0\),
\[
\begin{aligned}
5(k^2+h^2)-(k^2+h^2+2h+2)
&=4k^2+4h^2-2h-2 \\[1mm]
&=4k^2+2(2h+1)(h-1)
\geq0.
\end{aligned}
\]
Therefore,
\[
\begin{aligned}
J_h
&\geq\frac{1}{5(h^2+k^2)}
      \int_{h-1}^{h+1}t\,dt  =\frac{2h}{5(h^2+k^2)}
 \geq\frac{h}{5(h^2+k^2)}.
\end{aligned}
\]
Equivalently,
\[
\frac{h}{h^2+k^2}\leq5J_h.
\]
Combining the two estimates gives
\begin{equation}\label{A30}
\frac{J_h}{4}
\leq\frac{h}{h^2+k^2}
\leq5J_h.
\end{equation}

For
\[
h_j=b+2j+1,\qquad 0\leq j\leq N-1,
\]
the corresponding intervals are
\[
[h_j-1,h_j+1]
 =[b+2j,b+2j+2].
\]
These intervals have disjoint interiors and their union is
\[
[b,b+2N]=[b,m+n].
\]
Consequently, summing \eqref{A30} over \(j\) yields
\[
\sum_{j=0}^{N-1}
 \frac{b+2j+1}{(b+2j+1)^{2}+k^{2}}
 \asymp
 \int_b^{m+n}
 \frac{t}{1+k^{2}+t^{2}}\,dt.
\]
The integral can be computed explicitly:
\[
\begin{aligned}
\int_b^{m+n}
 \frac{t}{1+k^{2}+t^{2}}\,dt
 &=\frac12
   \log\frac{1+k^{2}+(m+n)^{2}}
              {1+k^{2}+b^{2}} =\frac12
   \log\frac{1+k^{2}+(m+n)^{2}}
              {1+k^{2}+(m-n)^{2}}.
\end{aligned}
\]
Since
\[
(m+n)^{2}=(m-n)^{2}+4mn,
\]
we obtain
\[
\begin{aligned}
\int_b^{m+n}
 \frac{t}{1+k^{2}+t^{2}}\,dt
 &=\frac12\log\left(
     1+\frac{4mn}
     {1+k^{2}+(m-n)^{2}}
   \right) =\frac12\log\left(
     1+\frac{4x_{1}y_{1}}
     {1+|x-y|^{2}}
   \right).
\end{aligned}
\]
For every \(v\geq0\),
\[
\log(1+v)
 \leq\log(1+4v)
 \leq4\log(1+v).
\]
Combining this inequality with \eqref{A29}, we obtain constants
\(c,C>0\) such that
\[
c\log\left(
1+\frac{x_{1}y_{1}}{1+|x-y|^{2}}
\right)
\leq \Phi_{2,+}(x,y)
\leq
C\log\left(
1+\frac{x_{1}y_{1}}{1+|x-y|^{2}}
\right).
\]
This proves \eqref{eq:two-green-logarithmic}.

Finally, estimate \eqref{eq:two-green-logarithmic} with \(y=e_1\) gives
\begin{equation}\label{y=e1}
\Phi_{2,+}(x,e_1)
 \asymp
 \log\left(
  1+\frac{x_{1}}{1+|x-e_1|^{2}}
 \right).
\end{equation}
Put
\[
t_x=\frac{x_{1}}{1+|x-e_1|^{2}}.
\]
Since
\[
x_{1}
 =1+(x_{1}-1)
 \leq1+|x-e_1|,
\]
we have
\[
0\leq t_x
 \leq\frac{1+|x-e_1|}{1+|x-e_1|^{2}}
 \leq C.
\]
Thus \(t_x\) ranges over a fixed bounded interval. On every bounded
interval \([0,M]\), there are constants \(c_M,C_M>0\) such that
\[
c_M t\leq\log(1+t)\leq C_Mt,
\qquad 0\leq t\leq M.
\]
Applying this to \eqref{y=e1}, we find
\begin{equation}\label{tag{A.32}}
\Phi_{2,+}(x,e_1)
 \asymp
 \frac{x_{1}}{1+|x-e_1|^{2}}.
\end{equation}
Moreover, the triangle inequality gives
\[
1+|x-e_1|\asymp1+|x|.
\]
Hence
\[
1+|x-e_1|^{2}\asymp(1+|x|)^{2},
\]
and therefore
\[
\Phi_{2,+}(x,e_1)
 \asymp
 \frac{x_{1}}{(1+|x|)^{2}}.
\]
Since \(0<\Phi_{2,+}(e_1,e_1)<\infty\) is a fixed constant, division by
\(\Phi_{2,+}(e_1,e_1)\) gives
\[
\frac{\Phi_{2,+}(x,e_1)}{\Phi_{2,+}(e_1,e_1)}
 \asymp
 \frac{x_{1}}{(1+|x|)^{2}},
\]
which is \eqref{eq:two-pole-appendix}.

For completeness, fix \(y\in\mathbb Z_{+}^{2}\). Since
\[
x_1
\leq |x_1-y_1|+y_1
\leq |x-y|+y_1,
\]
it follows from \eqref{eq:two-green-logarithmic} that
\begin{align*}
\Phi_{2,+}(x,y)
&\leq C\frac{x_1y_1}{1+|x-y|^2} \leq Cy_1
\left(
\frac{|x-y|}{1+|x-y|^2}
+\frac{y_1}{1+|x-y|^2}
\right).
\end{align*}
Besides, since
\[
\frac{|x-y|}{1+|x-y|^2}\leq \frac{2}{1+|x-y|}
\]
and
\[
\frac{1}{1+|x-y|^2}\leq \frac{2}{(1+|x-y|)^2},
\]
we obtain
\[
\Phi_{2,+}(x,y)
\leq
Cy_1\left(
\frac{1}{1+|x-y|}
+\frac{y_1}{(1+|x-y|)^2}
\right).
\]
Since \(y\) is fixed, \(|x-y|\to\infty\) as
\(|x|\to\infty\). Consequently,
\[
\Phi_{2,+}(x,y)\rightarrow0
\qquad\text{as } |x|\rightarrow\infty.
\]
The proof is complete.
\end{proof}

  \bigskip
  \medskip

 \small
\noindent {\bf  Conflicts of interest:} The authors declare that they have no conflicts of interest regarding this work.
\medskip

\noindent {\bf  Data availability:}  This paper has no associated data.\medskip

\noindent{\bf Acknowledgements:}
 H. Chen is supported by  the National Natural Science Foundation of China, No.  12361043.
\end{document}